\documentclass[12pt]{amsart}

\usepackage{amsfonts, amsthm, amsmath, amssymb}

\usepackage{amscd}

\usepackage[latin2]{inputenc}

\usepackage{t1enc}

\usepackage[mathscr]{eucal}

\usepackage{indentfirst}

\usepackage{graphicx}

\usepackage{graphics}

\usepackage{pict2e}

\usepackage{mathrsfs}

\usepackage{framed}

\newcommand{\tabincell}[2]{\begin{tabular}{@{}#1@{}}#2\end{tabular}}

\usepackage{mathabx}
\usepackage{booktabs}
\usepackage{makecell}
\usepackage{enumerate}
\usepackage[pagebackref]{hyperref}
\hypersetup{colorlinks=true}
\usepackage{cite}
\usepackage{color}
\usepackage{epic}
\usepackage{hyperref} %this gives clickable references, which is %nice
\numberwithin{equation}{section}
\theoremstyle{plain}

\usepackage{tikz}

\usetikzlibrary{calc}
\usetikzlibrary{shapes.geometric}

\tikzset{
  c/.style={every coordinate/.try}
}

\usetikzlibrary{arrows,shapes,positioning}
\usetikzlibrary{decorations.markings}
\tikzstyle arrowstyle=[scale=1]
\tikzstyle directed=[postaction={decorate,decoration={markings,mark=at position 0.6 with {\arrow[arrowstyle]{stealth};}}}]
\tikzstyle reverse directed=[postaction={decorate,decoration={markings,mark=at position 0.4 with {\arrowreversed[arrowstyle]{stealth};}}}]
\tikzstyle dot=[style={circle,inner sep=1pt,fill}]

\newtheorem{theorem}{Theorem}[section]

\newtheorem{lemma}[theorem]{Lemma}

\newtheorem{corollary}[theorem]{Corollary}

\newtheorem{proposition}[theorem]{Proposition}

\newtheorem{Fact}[theorem]{Fact}

\theoremstyle{definition}

\newtheorem{Def}[theorem]{Definition}

\newtheorem{example}[theorem]{Example}

\newtheorem{remark}[theorem]{Remark}

\newtheorem{?}[theorem]{Problem}

\newcommand{\Z}{\mathbb{Z}}
\def\FV{\mathrm{FV}}
\def\FZ{\mathrm{FZ}}
\def\YZL{\mathrm{YZL}}
\def\CSZ{\mathrm{CSZ}}
\def\SZ{\mathrm{SZ}}
\def\SS{\mathfrak{S}}
\def\cc{\mathrm{c}}
\def\r{\mathrm{r}}

\def\i{\mathrm{i}}
\def\rci{\mathrm{rci}}

\def\LH{\mathfrak{L}}
\def\N{\mathrm{N}}
\def\S{\mathrm{S}}
\def\E{\mathrm{E}}
\def\dE{\mathrm{dE}}
\def\NE{\mathrm{NE}}
\def\SE{\mathrm{SE}}
\def\SdE{\mathrm{SdE}}
\def\NdE{\mathrm{NdE}}
\def\cs{\mathrm{cs}}
\def\neb{\mathrm{neb}}
\def\nea{\mathrm{nea}}
\def\sdeb{\mathrm{sdeb}}
\def\sdea{\mathrm{sdea}}
\def\nde{\mathrm{nde}}

\def\het{\mathrm{ht}}
\def\wt{\mathrm{wt}}
\def\Neb{\mathrm{Neb}}
\def\Nea{\mathrm{Nea}}
\def\Sdeb{\mathrm{Sdeb}}
\def\Sdea{\mathrm{Sdea}}
\def\Ndeb{\mathrm{Ndeb}}
\def\Ndea{\mathrm{Ndea}}
\def\Nde{\mathrm{Nde}}

\def\Het{\mathrm{Ht}}
\def\Wt{\mathrm{Wt}}
\def\Asc{\mathrm{Asc}}
\def\asc{\mathrm{asc}}

\def\des{\mathrm{des}}
\def\ides{\mathrm{ides}}
\def\Des{\mathrm{Des}}
\def\Db{\mathrm{Db}}
\def\Ab{\mathrm{Ab}}
\def\Id{\mathrm{Ides}}
\def\Dt{\mathrm{Dt}}
\def\Ddif{\mathrm{Ddif}}
\def\ddif{\mathrm{ddif}}
\def\Dbot{\mathrm{Dbot}}
\def\dbot{\mathrm{dbot}}
\def\Dtb{\mathrm{Dtb}}
\def\Dta{\mathrm{Dta}}
\def\Dbb{\mathrm{Dbb}}
\def\Dba{\mathrm{Dba}}
\def\Abb{\mathrm{Abb}}
\def\Aba{\mathrm{Aba}}
\def\Lpk{\mathrm{Lpk}}
\def\Lval{\mathrm{Lval}}
\def\Lda{\mathrm{Lda}}
\def\Ldd{\mathrm{Ldd}}
\def\lpk{\mathrm{lpk}}
\def\lval{\mathrm{lval}}
\def\lda{\mathrm{lda}}
\def\ldd{\mathrm{ldd}}
\def\pprs{\mathrm{2\underline{13}}}
\def\ppwrs{\mathrm{2\underline{31}}}
\def\pp31-2{\mathrm{\underline{31}2}}
\def\rightascent{{\mathbf{2}\underline{13}}}
\def\leftdescent{\mathrm{\underline{31}}\mathbf{2}}
\def\rightdescent{\mathrm{\mathbf{2}\underline{31}}}

\def\Edif{\mathrm{Edif}}
\def\edif{\mathrm{edif}}
\def\Excp{\mathrm{Ep}}
\def\Exc{\mathrm{Exc}}
\def\Nexc{\mathrm{Nexc}}
\def\Nexcp{\mathrm{Nep}}
\def\Excpb{\mathrm{Epb}}
\def\Excpa{\mathrm{Epa}}
\def\Nexcb{\mathrm{Nexcb}}
\def\Nexca{\mathrm{Nexca}}
\def\Excb{\mathrm{Excb}}
\def\Exca{\mathrm{Exca}}
\def\Ine{\mathrm{Ine}}
\def\ine{\mathrm{ine}}

\def\exc{\mathrm{exc}}
\def\Ebot{\mathrm{Ebot}}
\def\ebot{\mathrm{ebot}}
\def\Cpk{\mathrm{Cpk}}
\def\Cval{\mathrm{Cval}}
\def\Cda{\mathrm{Cda}}
\def\Cdd{\mathrm{Cdd}}

\def\Scpk{\mathrm{Scpk}}
\def\Scval{\mathrm{Scval}}
\def\Scda{\mathrm{Scda}}
\def\Scdd{\mathrm{Scdd}}
\def\nep{\mathrm{nep}}
\def\nexc{\mathrm{nexc}}

\def\Vnexpb{\mathrm{Vnepb}}
\def\Vnexpa{\mathrm{Vnepa}}
\def\Vnexb{\mathrm{Vnexb}}
\def\Vnexa{\mathrm{Vnexa}}
\def\Vexpb{\mathrm{Vepb}}
\def\Vexpa{\mathrm{Vepa}}
\def\Vedif{\mathrm{Vedif}}
\def\Vbot{\mathrm{Vbot}}
\def\Vnex{\mathrm{Vnex}}
\def\Vnest{\mathrm{Vnest}}

\def\vbot{\mathrm{vbot}}
\def\vedif{\mathrm{vedif}}
\def\pone{\mathrm{pone}}
\def\vnest{\mathrm{vnest}}

\def\Inv{\mathrm{Inv}}
\def\Mak{\mathrm{Mak}}
\def\Mad{\mathrm{Mad}}
\def\Den{\mathrm{Den}}
\def\Makl{\mathrm{Makl}}
\def\Madl{\mathrm{Madl}}
\def\FZthree{\mathrm{FZ3}}
\def\FZfour{\mathrm{FZ4}}
\def\YZLone{\mathrm{YZL1}}
\def\YZLtwo{\mathrm{YZL2}}
\def\YZLthree{\mathrm{YZL3}}
\def\YZLfour{\mathrm{YZL4}}

\def\inv{\mathrm{inv}}
\def\den{\mathrm{den}}
\def\mak{\mathrm{mak}}
\def\mad{\mathrm{mad}}
\def\mak{\mathrm{mak}}
\def\inv{\mathrm{inv}}
\def\makl{\mathrm{makl}}
\def\madl{\mathrm{madl}}
\def\fzthree{\mathrm{fz3}}
\def\fzfour{\mathrm{fz4}}
\def\yzlone{\mathrm{yzl1}}
\def\yzltwo{\mathrm{yzl2}}
\def\yzlthree{\mathrm{yzl3}}
\def\yzlfour{\mathrm{yzl4}}
\def\maj{\mathrm{maj}}
\def\bast{\mathrm{bast}}
\def\foze{\mathrm{foze}}
\def\sist{\mathrm{sist}}
\def\sort{\mathrm{sor}}

\def\sist{\mathrm{sist}}
\def\nest{\mathrm{nest}}

\def\boxit#1{\leavevmode\hbox{\vrule\vtop{\vbox{\kern.33333pt\hrule
    \kern1pt\hbox{\kern1pt\vbox{#1}\kern1pt}}\kern1pt\hrule}\vrule}}

\usepackage{collectbox}



\begin{document}

\title[An involution on restricted Laguerre histories and its
applications]{An involution on restricted Laguerre histories and its
applications}

\author[J.~N.~Chen]{Joanna N. Chen}
\address{College of Science, Tianjin University of Technology, Tianjin 300384, P.R. China.}
\email{joannachen@tjut.edu.cn}

\author[S.~Fu]{Shishuo Fu}
\address{College of Mathematics and Statistics, Chongqing University, Chongqing 401331, P.R. China.}
\email{fsshuo@cqu.edu.cn}

\date{\today}

\begin{abstract}
Laguerre histories (restricted or not) are certain weighted Motzkin paths with two types of level steps. They are, on one hand, in natural bijection with the set of permutations, and on the other hand, yield combinatorial interpretations for the moments of Laguerre polynomials via Flajolet's combinatorial theory of continued fractions. In this paper, we first introduce a reflection-like involution on restricted Laguerre histories. Then, we demonstrate its power by composing this involution with three bijections due to Fran\c con-Viennot,  Foata-Zeilberger, and Yan-Zhou-Lin, respectively. A host of equidistribution results involving various (multiset-valued) permutation statistics follow from these applications. As byproducts, seven apparently new Mahonian statistics present themselves; new interpretations of known Mahonian statistics are discovered as well. Finally, in our effort to show the interconnections between these Mahonian statistics, we are naturally led to a new link between the variant Yan-Zhou-Lin bijection and the Kreweras complement.
\end{abstract}


\maketitle

%\tableofcontents

%%%%%%%%%%%%%%%%%%%%%%%%%%%%%%%%%%%%%
\section{Introduction}\label{sec:intro}
%%%%%%%%%%%%%%%%%%%%%%%%%%%%%%%%%%%%%

One interesting aspect of the orthogonal polynomials is the study of their moment sequences. For example, building on Flajolet's celebrated combinatorial theory of Jacobi continued fractions \cite{Fla}, Viennot has detailed in his memoir \cite{Vie} the links between various classical orthogonal polynomials (such as Hermite, Charlier, and Laguerre polynomials) and classical sequences in combinatorics. Such connections are valuable as they bear combinatorial information which naturally lead to refinements and generalizations of the enumerative results for the combinatorial objects involved. See for example \cite{CSZ,Cor,FV,FZ90,SS,KZ,SZ12,YZL,HMZ}, especially the two surveys by Corteel--Kim--Stanton \cite{CKS}, and Zeng \cite{Zen} and the references therein.

Sitting at the heart of Viennot's constructions was the notion of ``histories'', tailormade for each specific family of orthogonal polynomials. The main object considered in this paper is one kind of such histories, namely, the {\em (restricted) Laguerre histories}. Loosely speaking (see Section~\ref{sec:Preliminaries} for the precise definition), these are Motzkin paths where the east steps could come in two types (solid or dotted), and each step starting at height $h$ in the path is assigned a unique integer $i\in\{0,\ldots, h\}$.

Our main result, Theorem~\ref{thm:inv-xsi}, is a reflection-like involution constructed on the set of all shifted and restricted Laguerre histories (see Definition~\ref{def:LH}). This involution is subject to certain constraints, and consequently yields a wealth of information on various statistics. If we consider the following multivariate generating function
\begin{align}\label{gf:A}
A(z)&=A(t_1,t_2,t_3,t_4,r,s,x,v,w;z)\\
&:=\sum_{n\ge 1}\sum_{W\in\LH_n}t_1^{\sdeb\:W}t_2^{\sdea\:W}t_3^{\neb\:W}t_4^{\nea\:W}r^{\nde\:W}s^{\asc\:W}x^{\cs\:W}v^{\het\:W}w^{\wt\:W}z^n \nonumber\\
&=\sum_{n\ge 1}A_n(t_1,t_2,t_3,t_4,r,s,x,v,w)z^n. \nonumber
\end{align}
where $\LH_n$ stands for the set of all shifted and restricted Laguerre histories of length $n$ and the definitions of all the statistics will be given in the next section. Then a quick corollary of Theorem~\ref{thm:inv-xsi} is the following property enjoyed by $A(z)$.
\begin{corollary}\label{coro:gf of A}
\begin{align}\label{eq:A var change}
A(t_1,t_2,t_3,t_4,r,s,x,v,w;z) &=
\frac{x}{rs} A(t_4,\frac{t_3}{vw},vwt_2,t_1,\frac{1}{r},\frac{1}{s},\frac{1}{x},v,w;rsxz).
\end{align}
\end{corollary}

Note that as we will show in Section~\ref{sec:app}, the nine-variate polynomial $A_n(t_1,\ldots,w)$ encompasses a host of familiar polynomials in the literature. Let $\SS_n$ be the set of permutations of $[n]:=\{1,\ldots,n\}$. We have for instance (again, all the statistics will be defined in Section~\ref{sec:app}),
\begin{itemize}
  \item $A_n(t,t,1,1,1,1,1,1,1)=\sum_{\pi\in\SS_n}t^{\des(\pi)}$,
  \item $A_n(t,t,1,1,1,s,1,1,1)=\sum_{\pi\in\SS_n}t^{\des(\pi)}s^{\ides(\pi)}$,
  \item $A_n(t,tp^{-1},1,p^{-1},1,1,1,q,pq^{-1})=\sum_{\pi\in\SS_n}t^{\des(\pi)}p^{\pprs(\pi)}q^{\pp31-2(\pi)}$,
  \item $\lim_{p\rightarrow 0}A_n(t,tp^{-1},1,p^{-1},1,1,1,q,pq^{-1})=\sum_{\pi\in\SS_n(213)}t^{\des(\pi)}q^{\pp31-2(\pi)}$
\end{itemize}
are the classical {\it Eulerian polynomial}, the {\it double Eulerian polynomial} studied by Carlitz--Roselle--Scoville \cite{CRS}, Gessel \cite{Bra08,Bra15}, and Lin \cite{Lin}, a {\it $(p,q)$-Eulerian polynomial} studied by Br\"and\'en \cite{Bra08}, and Shin--Zeng \cite{SZ12}, a {\it $(q,t)$-Catalan number} studied by Blanco--Petersen \cite{BP}, Lin--Fu \cite{LF}, and Fu--Tang--Han--Zeng \cite{FTHZ}, respectively.

The rest of the paper is organized as follows. In section~\ref{sec:Preliminaries}, we briefly review the background of Laguerre polynomials and their sequence of moments to motivate our study on Laguerre histories. We then prove our main Theorem~\ref{thm:inv-xsi} and several corollaries, including Corollary~\ref{coro:gf of A} above, in section~\ref{sec:involution}. After recalling and rephrasing three bijections between $\SS_n$ and $\LH_n$ and introducing the respective permutation statistics transformed by the three bijections, we derive in section~\ref{sec:app} many equidistribution results by composing these bijections with our involution $\xi$ constructed in Theorem~\ref{thm:inv-xsi}. In particular, apparently new Mahonian statistics present themselves naturally, connections with known Mahonian statistics are established as well; see subsection~\ref{subsec:Mahonian}, especially Table~\ref{Mahonian stats} for more details.
%We end the paper by suggesting directions for further study.

%%%%%%%%%%%%%%%%%%%%%%%%%%%%%%%%%%%%%
\section{Background and Preliminaries}\label{sec:Preliminaries}
%%%%%%%%%%%%%%%%%%%%%%%%%%%%%%%%%%%%%
% In this section, we will recall the Laguerre polynomials, their moments, and three types of closely related histories, namely, the Laguerre histories, the restricted Laguerre histories and our shifted and restricted Laguerre histories.

The {\it Laguerre polynomials} \cite[Sect.~6.2]{AAR} are orthogonal with respect to the gamma distribution $x^{\alpha}e^{-x}dx$, where $\alpha >-1$. When multiplied by $(-1)^n n!$, the {\it monic Laguerre polynomials} $L_n^{(\alpha)}(x)$ can be defined either using their generating function
\begin{align*} %\label{gf:LP}
\sum_{n=0}^{\infty}L_n^{(\alpha)}(x)\frac{t^n}{n!} &=(1+t)^{-\alpha-1}\exp(\frac{xt}{1+t}),
\end{align*}
or via the explicit formula
\begin{align*} %\label{explicit-LP}
L_n^{(\alpha)}(x)=\sum_{k=0}^{n}(-1)^{n-k}(\alpha+k+1)_{n-k}\binom{n}{k}x^k,
\end{align*}
where $(x)_n=x(x+1)\cdots(x+n-1)~(n\ge 1)$ is the rising factorial with $(x)_0=1$, and $\binom{n}{k}=\frac{n!}{k!(n-k)!}$ is the usual binomial coefficient. The linear functional $\mathcal{L}$ with respect to which $L_n^{(\alpha)}(x)$ are orthogonal is defined by
\begin{align*}
\mathcal{L}(f)=\frac{1}{\Gamma(\alpha+1)}\int_{0}^{\infty}f(x)x^{\alpha}e^{-x}dx,
\end{align*}
thus $\mathcal{L}(x^n)=(\alpha+1)_n$ are the moments of $L_n^{(\alpha)}(x)$.

It is a classical result (see e.g.~\cite[Thm.~5.8.2]{AAR}) that if $\{p_n(x)\}$ is an orthogonal polynomial sequence (with respect to certain linear functional $\mathcal{L}$) satisfying the three-term recurrence
\begin{align*}
p_{n+1}(x)=(x-b_n)p_n(x)-\lambda_np_{n-1}(x), \text{ for } n\ge 1,
\end{align*}
then its moment-generating function has the following Jacobi continued fraction expansion
\begin{align}
\label{moment-gf and cf}
\sum_{n\ge 0}\mathcal{L}(x^n)z^n=\cfrac{1}{1-b_0z-\cfrac{\lambda_1z^2}{1-b_1z-\cfrac{\lambda_2z^2}{1-b_2z-\cfrac{\lambda_3z^2}{\ddots}}}}.
\end{align}

In the case of monic Laguerre polynomials, we have $b_n=2n+\alpha+1$ and $\lambda_n=n(n+\alpha)$. The next two special cases are the most interesting.
\begin{itemize}
  \item[$\alpha=1$:] the moment is $\mathcal{L}(x^n)=(n+1)!$, with $b_n=2n+2$, $\lambda_n=n(n+1)$;
  \item[$\alpha=0$:] the moment is $\mathcal{L}(x^n)=n!$, with $b_n=2n+1$, $\lambda_n=n^2$.
\end{itemize}

To understand \eqref{moment-gf and cf} combinatorially for $\alpha=1$ (resp.~$\alpha=0$), it suffices to set up a bijection from $\SS_{n+1}$ (resp.~$\SS_n$), which has cardinality $(n+1)!$ (resp.~$n!$), to the set of certain weighted $2$-Motzkin paths of length $n$ called Laguerre histories (resp.~restricted Laguerre histories), which are generated by the continued fraction on the right side of \eqref{moment-gf and cf} according to the Flajolet--Viennot theory. Such a bijection was first constructed by Fran\c con--Viennot \cite{FV}. This will be elaborated on in Section~\ref{sec:app}. We now give the precise definitions of the histories.

A {\em Motzkin path} of length $n$ is a lattice path in the first quadrant starting from $(0,0)$, ending at $(n,0)$, with three possible steps: northeast steps ($\N$) going from $(a,b)$ to $(a+1,b+1)$; southeast steps ($\S$) going from $(a,b)$ to $(a+1,b-1)$; and east steps ($\E$) going from $(a,b)$ to $(a+1,b)$. A {\em $2$-Motzkin path} is a Motzkin path whose east steps could come in two types: solid ($\E$) or dotted ($\dE$). The $2$-Motzkin paths will be represented as words consisted of letters from $\{\N,\S,\E,\dE\}$. For convenience, we will also label a step as $\NE$ (resp.~$\SdE$), if it is either a northeast (resp.~southeast) step or a solid (resp.~dotted) east step. The labels $\NdE$ and $\SE$ will be used similarly if needed.

A {\em Laguerre history}  of length $n$ is a pair $(w,c)$ such that
$w=w_1 \cdots w_n$ is a 2-Motzkin path and $c=(c_1,\cdots, c_n)$
is a choice vector with $0 \leq c_i \leq h_i$, where
\[h_i := \#\{j : j<i, w_j=\N\}-\#\{j : j<i, w_j=\S\}\]
is the height of the $i$-th step of $w$.
A {\em restricted  Laguerre history} is usually defined as
the pair $(w,c)$ such that
$$0 \leq  c_i \leq \begin{cases}
  h_i, & \text{if } w_i=\NE,\\
  h_i-1, & \text{if } w_i=\SdE.
  \end{cases}$$
Let $\mathtt{L}_n$ (resp.~$\mathtt{L}_n^*$) be the set of
Laguerre histories (resp.~restricted Laguerre histories) of length $n$.
It is well known that
the cardinality of $\mathtt{L}_n$ (resp.~$\mathtt{L}_n^*$) is $(n+1)!$ (resp.~$n!$), so it corresponds to the Laguerre polynomial $L_n^{(1)}(x)$ (resp.~$L_n^{(0)}(x)$).

As it turns out, our involution is better explained using the following shifted version of restricted Laguerre histories.

\begin{Def}\label{def:LH}
A {\it (shifted and restricted) Laguerre history} of length $n$ is a triple $W=(w,h,c)$, such that
\begin{enumerate}
  \item $w=w_1\cdots w_n$ is a $2$-Motzkin path of length $n$.
  \item $h=h_1\cdots h_n$ with $h_i:=\#\{j : j<i, w_j=\N\}-\#\{j : j<i, w_j=\S\}$.
  \item $c=c_1\cdots c_n$ with $h_i\ge c_i\ge \begin{cases}
  0, & \text{if } w_i=\NE,\\
  1, & \text{if } w_i=\SdE.
  \end{cases}$
\end{enumerate}
For each $i=1,2,\ldots,n$, we say the $i$-th step of $W$ is of type $w_i$, height $h_i$, and weight $c_i$.
\end{Def}

\begin{remark}
Note that our choice for the bounds in (3) is different (shifted by $1$ for $\SdE$ steps) from the definition of restricted Laguerre histories, and it forces east steps on the $x$-axis to be all solid. Despite of the fact that $h$ is completely determined by $w$, we prefer the triple $(w,h,c)$ to the pair $(w,c)$ in our definition, so that the role played by $h$ is emphasized and it makes our later construction clearer.
\end{remark}

In what follows, we will abbreviate ``shifted and restricted Laguerre histories'' as ``sr-Laguerre histories''. The set of all sr-Laguerre histories of length $n$ is denoted as $\LH_n$. Clearly, we have $|\LH_n|=|\mathtt{L}_n^*|=n!$. See Fig.~\ref{fig:reslagu} for a concrete example of sr-Laguerre history.

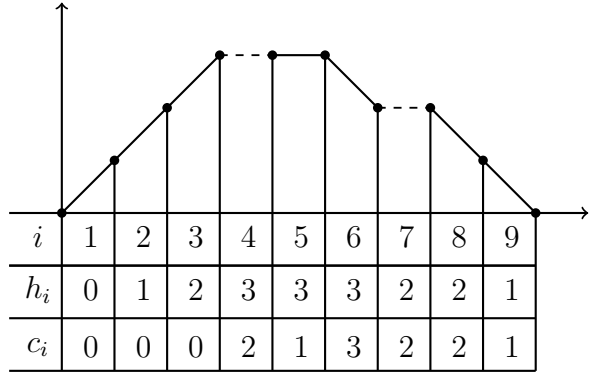
\begin{figure}[!htbp]
\begin{center}
\begin{tikzpicture}[line width=0.7pt,scale=0.7]
\coordinate (O) at (0,0);
\draw[thick][->] (O)++(-1,0)--++(11,0);
\draw[thick][->] (O)++(0,-3)--++(0,7);
\draw[thick] (O)--++(3,3)++(1,0)--++(1,0)--++(1,-1)++(1,0)--++(2,-2);
\draw[dashed](O)++(3,3)--++(1,0)++(2,-1)--++(1,0);

\fill[black!100] (O) circle(0.5ex) ++(1,1)circle(0.5ex)
++(1,1)circle(0.5ex) ++(1,1)circle(0.5ex)++(1,0)circle(0.5ex)
++(1,0)circle(0.5ex)++(1,-1)circle(0.5ex)++(1,0)circle(0.5ex)
++(1,-1)circle(0.5ex)++(1,-1)circle(0.5ex);
\draw[thick]  (O)++(-1,-1)--++(10,0);
\draw[thick]  (O)++(-1,-2)--++(10,0);
\draw[thick]  (O)++(-1,-1)--++(10,0);
\draw[thick]  (O)++(-1,-3)--++(10,0);
\draw[thick]  (O)++(1,-3)--++(0,4);
\draw[thick]  (O)++(2,-3)--++(0,5);
\draw[thick]  (O)++(3,-3)--++(0,6);
\draw[thick]  (O)++(4,-3)--++(0,6);
\draw[thick]  (O)++(5,-3)--++(0,6);
\draw[thick]  (O)++(6,-3)--++(0,5);
\draw[thick]  (O)++(7,-3)--++(0,5);
\draw[thick]  (O)++(8,-3)--++(0,4);
\draw[thick]  (O)++(9,-3)--++(0,3);

\path (-0.45,-0.45) node {$i$} ++(1,0) node {$1$} ++(1,0) node {$2$} ++(1,0) node {$3$} ++(1,0) node {$4$} ++(1,0) node {$5$} ++(1,0) node {$6$} ++(1,0) node {$7$} ++(1,0) node {$8$} ++(1,0) node {$9$};
\path (-0.45,-1.45) node {$h_i$} ++(1,0) node {$0$} ++(1,0) node {$1$} ++(1,0) node {$2$} ++(1,0) node {$3$} ++(1,0) node {$3$} ++(1,0) node {$3$} ++(1,0) node {$2$} ++(1,0) node {$2$} ++(1,0) node {$1$};
\path (-0.45,-2.55) node {$c_i$} ++(1,0) node {$0$} ++(1,0) node {$0$} ++(1,0) node {$0$} ++(1,0) node {$2$} ++(1,0) node {$1$} ++(1,0) node {$3$} ++(1,0) node {$2$} ++(1,0) node {$2$} ++(1,0) node {$1$};
\end{tikzpicture}
\caption{An sr-Laguerre history of length $9$.}
\label{fig:reslagu}
\end{center}
\end{figure}

%%%%%%%%%%%%%%%%%%%%%%%%%%%%%%%%%%%%%%%%%%
\section{An involution on sr-Laguerre histories}\label{sec:involution}
%%%%%%%%%%%%%%%%%%%%%%%%%%%%%%%%%%%%%%%%%%
We are going to construct an involution over $\LH_n$. Given a permutation $\pi=\pi(1)\cdots\pi(n)\in\SS_n$, its last entry $\pi(n)$ can sometimes be viewed as an auxiliary statistic. In particular, it will help us link $(2\underline{31})\pi$ with $(2\underline{13})\pi$ (see Section~\ref{subsec:FV} for the definition of vincular patterns $2\underline{31}$ and $2\underline{13}$). The following definition gives a counterpart of ``last entry'' over the sr-Laguerre histories, which will be crucial for our later use.

\begin{Def}
Given an sr-Laguerre history $W=(w,h,c)\in\LH_n$, we define its {\it critical step} to be the last step of $W$ that is weighted by $0$. We use
\begin{align}
\cs(W)& :=\max\{i\in[n] : c_i=0\}
\end{align}
to denote the index of the critical step of $W$.
\end{Def}

Note that the first step of any Laguerre history must have height $0$, thus is weighted by $0$, so $\cs(W)$ is well-defined. Now we can state the main result of this paper.

\begin{theorem}\label{thm:inv-xsi}
There exists a unique involution $\xi:\LH_n\rightarrow \LH_n$ such that, if $W=(w,h,c)$ with $\cs(W)=m$ and $V:=\xi(W)=(v,g,b)$, then
\begin{enumerate}
	\item $\cs(V)=n+1-m$.
	\item For any $n+1-m\neq j\in [n]$, $v_j=\NE$ if and only if $w_{n+1-j}=\SdE$.
	\item For any $j\in[n]$, $g_j=\begin{cases}
	h_{n+1-j}+1, & \text{if $j>n+1-m$ and $v_j=\SdE$,} \\
	h_{n+1-j}-1, & \text{if $j<n+1-m$ and $v_j=\NE$,} \\
	h_{n+1-j}, & \text{otherwise}.
	\end{cases}$
	\item For any $j\in[n]$, $b_j=g_j-h_{n+1-j}+c_{n+1-j}$.
\end{enumerate}
\end{theorem}
\begin{proof}
To determine the $j$-th step $(v_j,g_j,b_j)$ of $V=\xi(W)$ for $j=1,2,\ldots,n$, we examine the steps $(w_{n+1-j},h_{n+1-j},c_{n+1-j})$ and $(w_{n-j},h_{n-j},c_{n-j})$ (for $j<n$) of $W$. Depending on the value of $j$ compared with $n+1-m$ (i.e., $j=n+1-m$, $j<n+1-m$, or $j>n+1-m$), the four conditions (1)--(4) collectively decide $(v_j,g_j,b_j)$. Namely, condition (1) or (2), together with the type of $w_{n+1-j}$, decide the type of $v_j$ to be either $\NE$ or $\SdE$; condition (3), together with $h_{n+1-j}$ and $h_{n-j}$, decide the values of $g_j$ and $g_{j+1}$, hence the exact type of $v_j$ (consider $g_{j+1}-g_j$); condition (4) then yields the value of $b_j$. The triple $(v_j,g_j,b_j)$ derived this way is clearly unique. All it remains is to show that: (i) the four conditions are compatible with each other; (ii) the triple $(v,g,b)$ thus obtained is indeed an sr-Laguerre history of length $n$. We elaborate on these two claims below.
\begin{enumerate}[(i)]
	\item More precisely, we have to show that if condition (1) or (2) tells us $v_j=\NE$ (resp.~$v_j=\SdE$), then condition (3) should yield $g_{j+1}-g_j = 0$ or $1$ (resp.~$g_{j+1}-g_j = 0$ or $-1$). This is indeed the case, as can be verified with Table~\ref{xsi-cases}, wherein the values from the $4$th to the $9$th column are enforced by conditions (1)--(4). Note that in the last two cases with $j=n$, we have set $g_{j+1}=0$ to indicate that the path should end on the $x$-axis.
	\item We have to verify all three conditions in Definition~\ref{def:LH}.
	\begin{itemize}
		\item[(ii-1)] For condition (1), we use Table~\ref{xsi-cases} to check that $g_1=g_{n+1}=0$, and when $v_j=\SdE$, we always have $g_j\ge 1$. For the latter, we detail on one such case. Suppose $(j,v_j,v_{j+1})=(n-m,\SdE,\NE)$, then $g_j=h_{n+1-j}$ by the 4th row of Table~\ref{xsi-cases}. But now $n+1-j=m+1>m$, making $w_{n+1-j}=\NE$ a step after the critical step of $W$, which means $g_j=h_{n+1-j}\ge 1$, as desired.
		\item[(ii-2)] Condition (2) is implicitly satisfied when we make our choice for the exact type of $v_j$ according to the value of $g_{j+1}-g_j$. Namely, when $g_{j+1}-g_j=1$ (resp.~$0$, $-1$), we choose $v_j=\N$ (resp.~$\E$ or $\dE$, $\S$).
		\item[(ii-3)] Condition (3) can be verified case-by-case using the range of $b_j$ as shown in the last column of Table~\ref{xsi-cases}.
	\end{itemize}
\end{enumerate}

Finally, suppose $(w,h,c)\stackrel{\xi}{\rightarrow}(v,g,b)\stackrel{\xi}{\rightarrow}(\tilde{w},\tilde{h},\tilde{c})$. To see that $\xi$ is an involution, we need to show that $(\tilde{w}_j,\tilde{h}_j,\tilde{c}_j)=(w_j,h_j,c_j)$ for each $j=1,2,\ldots,n$. Noting that the height sequence $h$ completely determines the sequence $w$ of step types, and that $\tilde{h}_j-\tilde{c}_j=g_{n+1-j}-b_{n+1-j}=h_j-c_j$ according to condition (4), we see that it suffices to prove $\tilde{h}_j=h_j$. This needs to be checked for each of the five cases: 1) $j=m$; 2) $j<m$ and $w_j=\NE$; 3) $j<m$ and $w_j=\SdE$; 4) $j>m$ and $w_j=\NE$; 5) $j>m$ and $w_j=\SdE$. We trust the readers to fill in the details. So we see $(\tilde{w},\tilde{h},\tilde{c})=\xi^2(w,h,c)=(w,h,c)$, and $\xi$ is indeed an involution.
\end{proof}

\begin{table}
{\small
\begin{tabular}{ccccccccc}
\toprule
$j$ & $v_j$ & $v_{j+1}$ & $w_{n+1-j}$ & $w_{n-j}$ & $g_j$ & $g_{j+1}$ & $g_{j+1}-g_j$ & $b_j$\\
\midrule
$n+1-m$ & $\NE$ & $\NE$ & $\NE$ & $\SdE$ & $h_{n+1-j}$ & $h_{n-j}$ & 0 or 1 & 0 \\
&&&&&&&& \\
$n+1-m$ & $\NE$ & $\SdE$ & $\NE$ & $\NE$ & $h_{n+1-j}$ & $h_{n-j}+1$ & 0 or 1 & 0 \\
&&&&&&&& \\
$n-m$ & $\NE$ & $\NE$ & $\SdE$ & $\NE$ & $h_{n+1-j}-1$ & $h_{n-j}$ & 0 or 1 & $[0,g_j]$ \\
&&&&&&&& \\
$n-m$ & $\SdE$ & $\NE$ & $\NE$ & $\NE$ & $h_{n+1-j}$ & $h_{n-j}$ & 0 or $-1$ & $[1,g_j]$ \\
&&&&&&&& \\
$<n-m$ & $\NE$ & $\NE$ & $\SdE$ & $\SdE$ & $h_{n+1-j}-1$ & $h_{n-j}-1$ & 0 or 1 & $[0,g_j]$ \\
&&&&&&&& \\
$<n-m$ & $\NE$ & $\SdE$ & $\SdE$ & $\NE$ & $h_{n+1-j}-1$ & $h_{n-j}$ & 0 or 1 & $[0,g_j]$ \\
&&&&&&&& \\
$<n-m$ & $\SdE$ & $\NE$ & $\NE$ & $\SdE$ & $h_{n+1-j}$ & $h_{n-j}-1$ & 0 or $-1$ & $[1,g_j]$ \\
&&&&&&&& \\
$<n-m$ & $\SdE$ & $\SdE$ & $\NE$ & $\NE$ & $h_{n+1-j}$ & $h_{n-j}$ & 0 or $-1$ & $[1,g_j]$ \\
&&&&&&&& \\
$>n+1-m$ & $\NE$ & $\NE$ & $\SdE$ & $\SdE$ & $h_{n+1-j}$ & $h_{n-j}$ & 0 or 1 & $[1,g_j]$ \\
&&&&&&&& \\
$>n+1-m$ & $\NE$ & $\SdE$ & $\SdE$ & $\NE$ & $h_{n+1-j}$ & $h_{n-j}+1$ & 0 or 1 & $[1,g_j]$ \\
&&&&&&&& \\
$>n+1-m$ & $\SdE$ & $\NE$ & $\NE$ & $\SdE$ & $h_{n+1-j}+1$ & $h_{n-j}$ & 0 or $-1$ & $[1,g_j]$ \\
&&&&&&&& \\
$>n+1-m$ & $\SdE$ & $\SdE$ & $\NE$ & $\NE$ & $h_{n+1-j}+1$ & $h_{n-j}+1$ & 0 or $-1$ & $[1,g_j]$ \\
&&&&&&&& \\
$n \: (m=1)$ & $\E$ & n/a & $\NE$ & n/a & 0 & 0 & 0 & 0 \\
&&&&&&&& \\
$n \: (m>1)$ & $\S$ & n/a & $\NE$ & n/a & 1 & 0 & $-1$ & 1 \\
\bottomrule\\
\end{tabular}
}
\caption{All cases in the construction of $V=(v,g,b)=\xi(W)$}\label{xsi-cases}
\end{table}

Next, we introduce eight set-valued (or multiset-valued) statistics on sr-Laguerre histories, and discuss the implications of Theorem~\ref{thm:inv-xsi} in terms of these statistics.

\begin{Def}
For an sr-Laguerre history $W=(w,h,c)\in\LH_n$, we denote respectively the sets of $\NE$ steps before the critical step, $\SdE$ steps before the critical step, $\NdE$ steps before the critical step, $\NE$ steps after the critical step, $\SdE$ steps after the critical step, and $\NdE$ steps after the critical step as $\Neb(W)$, $\Sdeb(W)$, $\Ndeb(W)$, $\Nea(W)$, $\Sdea(W)$, and $\Ndea(W)$. Let
$$\Nde(W):=\{i\in[n-1]: w_i=\NdE\},
$$ and let $\Het(W)$ (resp.~$\Wt(W)$) be the multiset consisted of $h_i$ (resp.~$c_i$) copies of $i$ for $1 \leq i \leq n$.
\end{Def}

Throughout this paper, we use the convention that if ``St'' is a set-valued (or multiset-valued) statistic, then ``st'' is the corresponding numerical statistic. For example, $\het(W)$ is the cardinality of $\Het(W)$ (as a multiset) for each $W$. Moreover, since multisets are constantly involved in this paper, we prefer the disjoint union symbol ``$A\sqcup B$'' to the usual ``$A\cup B$''.

\begin{corollary}\label{coro:num only}
Given any $W\in\LH_n$, suppose $V=\xi(W)$. We have
\begin{align}
(\het-\wt,\neb,\sdeb,\nea,\sdea)\,W = (\het-\wt,\sdea,\nea,\sdeb,\neb)\,V.
\end{align}
\end{corollary}
\begin{proof}
Condition (4) in the construction of $\xi$ ensures that $\het(W)-\wt(W)=\het(V)-\wt(V)$, while the remaining four equalities follow from condition (2).
\end{proof}

The following definition leads to an sr-Laguerre history counterpart of the permutation statistic $\Id$, the set of inverse descents (see Section~\ref{subsec:FV}).

\begin{Def}
For an sr-Laguerre history $W\in\LH_n$, let $\Asc(W)$ be the
set containing all the indices $i\in[n-1]$ satisfying one of the following conditions:
\begin{enumerate}
	\item $c_i<c_{i+1}$, with $w_i=\NE$,
	\item $c_i\le c_{i+1}$, with $w_{i}=\SdE$.
\end{enumerate}
We call such $i$ an {\it ascent} of $W$.
\end{Def}

Fix a positive integer $n$, we introduce the following {\em complementation with respect to $n$}, which is denoted as $\kappa_n$ and defined to be a map sending a multiset to another multiset with the same cardinality. More precisely, let $S=\{s_1,s_2,\ldots,s_m\}$ be certain multiset, then $$\kappa_n(S):=\{n-s_1,n-s_2,\ldots,n-s_m\}.$$

Now we can derive the following set/multiset-valued extension of Corollary~\ref{coro:num only}.

\begin{corollary}\label{coro:xi-stats}
Given any $W\in\LH_n$, suppose $V=\xi(W)$. We have
\begin{align}
&(\Neb,\Sdeb,\Nea,\Sdea,\Het,\Wt)\:W = \notag\\
& \quad \kappa_{n+1}\circ(\Sdea,\Nea,\Sdeb,\Neb,\Het\sqcup\Neb\setminus\Sdea,\Wt\sqcup\Neb\setminus\Sdea)\:V,
\label{coro:xieq1}
\end{align}
Moreover, we have
\begin{align}
\label{coro:xieq2}
[n-1] \setminus \,\Nde(W) &= \kappa_n\circ\Nde(V),\\
\label{coro:xieq3}
[n-1] \setminus \,\Asc(W) &= \kappa_n\circ\Asc(V).
\end{align}
\end{corollary}
\begin{proof}
Condition (3) in the construction of $\xi$ ensures that
$$\Het(W)=\kappa_{n+1}(\Het\sqcup\Neb\setminus\Sdea(V)).$$
Condition (4) ensures that
$$\Het\setminus\Wt(W)=\kappa_{n+1}(\Het\setminus\Wt(V)) \text{ and } \Wt(W)=\kappa_{n+1}(\Wt\sqcup\Neb\setminus\Sdea(V)).$$
The remaining four equalities in (\ref{coro:xieq1}) follow from condition (2). Moreover, a case-by-case check using Table \ref{xsi-cases} reveals that $v_j=\NdE$ if and only if $w_{n-j}=\SE$ for $1 \leq j <n$, which proves \eqref{coro:xieq2}. Take the case in row $5$ as an example. If $v_j=\N$, implying that $g_{j+1}-g_j=1$, then $w_{n-j}=\S$, since $h_{n+1-j}-h_{n-j}=-(g_{j+1}-g_j)=-1$.

Equality (\ref{coro:xieq3}) follows from the fact that
$j \in \Asc(V)$ if and only if $n-j \notin \Asc(W)$ for $1 \leq j <n$.
This can be  checked through Table \ref{xsi-cases}.
As an example, we show it holds for the case in row $6$. Notice that $c_{n-j}=h_{n-j}-g_{j+1}+b_{j+1}=b_{j+1}$
and $c_{n+1-j}=h_{n+1-j}-g_{j}+b_{j}=b_{j}+1$. Then $b_j < b_{j+1}$
if and only if  $c_{n-j} \geq c_{n+1-j}$, as desired. Other cases can be verified similarly.
\end{proof}

Interpreting Corollary~\ref{coro:xi-stats} using the multivariate generating function $A(z)$ defined in \eqref{gf:A}, we can now prove Corollary~\ref{coro:gf of A}.
\begin{proof}[Proof of Corollary~\ref{coro:gf of A}]
Take any $W\in\LH_n$ and suppose $V=\xi(W)$. By Corollary~\ref{coro:xi-stats}, we have
\begin{align*}
& (\neb,\sdeb,\nea,\sdea,\het,\wt)\:W =\\
& \qquad (\sdea,\nea,\sdeb,\neb,\het+\neb-\sdea,\wt+\neb-\sdea)\:V,\\
& \nde(W)=n-1-\nde(V),\quad \asc(W)=n-1-\asc(V),\\
& \cs(W)=n+1-\cs(V).
\end{align*}
Plugging these relations back into the generating function of $A(z)$ and noting that as $W$ runs through all histories in $\LH_n$, so does $V$, we get \eqref{eq:A var change} and finish the proof.
\end{proof}

%%%%%%%%%%%%%%%%%%%%%%%%%%%%%%%%%%%%%%%%%%
\section{Application: equidistributions of permutation statistics}\label{sec:app}
%%%%%%%%%%%%%%%%%%%%%%%%%%%%%%%%%%%%%%%%%%

In this section, with the involution $\xi$ in mind, we first recall an existing bijection between permutations and (restricted) Laguerre histories, and then modify it to get a bijection between permutations and sr-Laguerre histories. Suppose $f: \SS_n\rightarrow \LH_n$ is such a bijection, then our strategy is to consider the compositional map $f^{-1}\circ\xi\circ f$, and see what equidistribution results concerning various permutation statistics can be deduced from this map.

In the following three subsections, we apply the above strategy to three bijections, namely, two classic bijections due to Fran\c con-Viennot \cite{FV},  Foata-Zeilberger \cite{FZ90}, respectively, and a most recent one due to Yan-Zhou-Lin \cite{YZL}. All the permutation statistics involved will be defined immediately before they are needed. In the final subsection \ref{subsec:Mahonian}, we start with four Mahonian statistics considered by Clarke-Steingr\'{i}msson-Zeng \cite{CSZ}, and look for their counterparts through the aforementioned compositional maps. The results are summarized in Table~\ref{Mahonian stats} and Theorem~\ref{thm:Mahon stats}, and further explored in Appendix~\ref{sec:appendix}.

%%%%%%%%%%%%%%%%%%%%%%%%%%%%%
\subsection{Application with \texorpdfstring{$\Phi_{\FV}$}{Fran\c con-Viennot's bijection}}\label{subsec:FV}

Fran\c con and Viennot's original bijection \cite[Theorem 2.2]{FV} was essentially from $\mathtt{L}_{n-1}$ to $\SS_n$, where permutations were bordered with $(-\infty,-\infty)$ and the interpretation of the weights $c_i$ is counting from left to right. The version we are going to introduce here is a variant $\Phi_{\FV}:\SS_n\rightarrow \LH_n$ with the boundary condition $(-\infty,\infty)$ and the counting for $c_i$ is from right to left.

We begin by recalling a few set-valued {\em linear}\footnote{Here ``linear'' means these statistics are defined using the one-line notation of permutations, as opposed to the ``cyclic'' ones, whose definitions are usually better understood using the cycle notation of permutations; see e.g., \cite{CSZ}.} permutation statistics. Given a permutation $\pi=\pi(1) \pi(2) \cdots \pi(n) \in \SS_n$, let
\begin{align*}
   \Des(\pi) &=\{i : \pi(i) > \pi(i+1), 1\le i<n\}, \quad \Dt(\pi)  =\{\pi(i) : i\in\Des(\pi)\},  \\[3pt]
  \Db(\pi) &=\{\pi(i+1) : i\in\Des(\pi)\},  \quad  \Ab(\pi) =\{\pi(i) : \pi(i)< \pi(i+1), 1\le i<n\},\; \text{and} \\[3pt]
  \Id(\pi)&=\Des(\pi^{-1})
\end{align*}
be the sets of {\em descents, descent tops, descent bottoms, ascent bottoms,} and {\em inverse descents} of $\pi$, respectively. Here $\pi^{-1}$ refers to the group-theoretical inverse of $\pi$. By our convention, $\des(\pi)$ (resp.~$\ides(\pi)$) is then understood to be the cardinality of $\Des(\pi)$ (resp.~$\Id(\pi)$) and called the {\em descent number} (resp.~{\em inverse descent number}) of $\pi$. Moreover, we need the multiset-valued statistics {\em descent difference} and {\em descent bottoms sum}\footnote{Note that in \cite[Definition 3]{CSZ}, $\Ddif$ and $\Dbot$ are merely numerical statistics, for which we use lowercase letters $\ddif$ and $\dbot$ to avoid confusion (see section \ref{subsec:Mahonian}). Moreover, $\Db(\pi)$ completely determines $\Dbot(\pi)$ and vice versa.}, defined as
\begin{align*}
\Ddif(\pi) &=\bigcup_{i\in\Des(\pi)}\{\pi(i+1)+1,\pi(i+1)+2,\ldots,\pi(i)\},\\
\Dbot(\pi) &=\bigcup_{i\in\Des(\pi)}\{\pi(i+1)^{\pi(i+1)}\},
\end{align*}
where $a^b$ means $b$ copies of $a$, and further refined statistics
\begin{align*}
  \Dtb(\pi)=\{\pi(i) : \pi(i) <\pi(n), \pi(i) \in \Dt(\pi)\}, &\,\,\,\, \Dta(\pi)=\{\pi(i) : \pi(i) >\pi(n), \pi(i) \in \Dt(\pi)\}, \\[3pt]
  \Dbb(\pi)=\{\pi(i) : \pi(i) <\pi(n), \pi(i) \in \Db(\pi)\}, &\,\,\,\, \Dba(\pi)=\{\pi(i) : \pi(i) >\pi(n), \pi(i) \in \Db(\pi)\}, \\[3pt]
  \Abb(\pi)=\{\pi(i) : \pi(i) <\pi(n), \pi(i) \in \Ab(\pi)\}, & \,\,\,\,\Aba(\pi)=\{\pi(i) : \pi(i) >\pi(n), \pi(i) \in \Ab(\pi)\}.
\end{align*}

% Let $\inv(\pi)$ be the inversion of $\pi$ as the number
% of inversion pairs $(i,j)$ such that $\pi_i> \pi_j$
% and $i < j$. It is well known that
%  statistics equidistributed with $\des$ is called Eulerian,
% while statistics equidistributed with $\inv$ are Mahonian.

Next we consider several statistics involving the notion of permutation patterns, which we briefly recall here for the sake of completeness. An occurrence of a classical pattern $p$ in a permutation $\sigma$ is a subsequence of $\sigma$ that is order-isomorphic to $p$. In 2000, Babson and Steingr\'{i}msson \cite{BS} generalized the notion of permutation patterns, to what are now known as vincular patterns; see the book exposition by Kitaev \cite{Kit}, as well as our recent work \cite{ChenFu} related to vincular patterns. Adjacent letters in a vincular pattern which are underlined must stay adjacent when they are placed back to the original permutation. For instance, $41253$ contains only one occurrence of the vincular pattern $\underline{31}42$ in its subsequence $4153$. Given a vincular pattern $\tau$ and a permutation $\pi$, we denote by $\tau(\pi)$ the number of occurrences of the pattern $\tau$ in $\pi$, and $(\tau_1+\tau_2)(\pi):=\tau_1(\pi)+\tau_2(\pi)$. We introduce the coordinate statistics
\begin{eqnarray*}
& \pprs_i (\pi) =  \# \{ ~j : i <j <n  ~\text{and}~  \pi(j) < \pi(i) < \pi(j+1)  \},  \\[3pt]
&\ppwrs_i (\pi) =  \# \{ ~j : i <j <n  ~\text{and}~  \pi(j+1) < \pi(i) < \pi(j)  \},  \\[3pt]
&\pp31-2_i (\pi) =  \# \{ ~j : 1  \leq j <i-1  ~\text{and}~  \pi(j+1) < \pi(i) < \pi(j)  \}
\end{eqnarray*}
for $1\le i\le n$, and the corresponding multiset-valued statistics
\begin{eqnarray*}
  &\rightascent(\pi) = \{\pi(i) : (\pi(i), \pi(j)) \text{ is a pair such that } i<j<n \text{ and } \pi(j) < \pi(i) < \pi(j+1) \}, \\[4pt]
   & \rightdescent(\pi) = \{\pi(i) : (\pi(i), \pi(j)) \text{ is a pair such that } i<j<n \text{ and } \pi(j+1) < \pi(i) < \pi(j) \}, \\[4pt]
  & \leftdescent(\pi) = \{\pi(i) : (\pi(j), \pi(i)) \text{ is a pair such that } 1 \leq j <i-1 \text{ and } \pi(j+1) < \pi(i) < \pi(j) \}.
\end{eqnarray*}

Clearly, $\pprs_i\, \pi$ is the number of occurrences of $\pi(i)$ in $\rightascent(\pi)$, thus, $\rightascent(\pi) =\rightascent(\sigma)$ if and only if we have
$\pprs_i(\pi)=\pprs_i(\sigma)$ for $1 \leq i \leq n$. The following example could be used to confirm all the linear statistics we have introduced so far.

% Throughout this paper, we make the convention that if $\st_i(w)$ is certain coordinate statistic defined for each $i\in[n]$, then the bold-faced $\bst(w):=(\st_1(w)~\cdots~\st_n(w))$ and $\st(w):=\sum_i\st_i(w)$ stand for the corresponding vector statistic and global statistic, respectively.

\begin{example}
For $\sigma=618742593$, we have
$\Des(\sigma)=\{1,3,4,5,8\}$,
$\Id(\sigma)=\{3,5,7\}$,
$\Dt(\sigma)=\Dta(\sigma)=\{4,6,7,8,9\}$,
$\Dtb(\sigma)=\emptyset$,
$\Db(\sigma)=\{1,2,3,4,7\}$,
$\Dbb(\sigma)=\{1,2\}$,
$\Dba(\sigma)=\{4,7\}$,
$\Ab(\sigma)=\{1,2,5\}$,
$\Abb(\sigma)=\{1,2\}$,
$\Aba(\sigma)=\{5\}$,
$\Ddif(\sigma)=\{2,3^2,4^3,5^3,\linebreak$
$6^3,7^2,8^2,9\}$,
$\Dbot(\sigma)=\{1,2^2,3^3,4^4,7^7\}$,
$\rightascent(\sigma)=\{4,6^2,7,8\}$,
$\rightdescent(\sigma)=\{4,5,6^2,7,8\}$,
$\leftdescent(\sigma)=\{2,3^2,4,5^2\}$.
\end{example}

Now we can present our variant of Fran\c con-Viennot's bijection, $\Phi_{\FV}: \SS_n \to \LH_n$.
Given $\pi=\pi(1)\pi(2) \cdots \pi(n) \in \SS_n$, let $\pi(0)=-\infty$ and
$\pi(n+1)=\infty$. For $1 \leq i \leq n$ we call $\pi(i)$ a
\begin{itemize}
  \item {\em linear peak} if $\pi(i-1) < \pi(i) >\pi(i+1)$;
  \item {\em linear valley} if $\pi(i-1) > \pi(i) < \pi(i+1)$;
  \item {\em linear double ascent} if $\pi(i-1) < \pi(i) <\pi(i+1)$;
  \item {\em linear double descent} if $\pi(i-1) > \pi(i) >\pi(i+1)$.
\end{itemize}
The set of linear peaks (resp.~valleys, double ascents, double descents) is denoted as $\Lpk$ (resp.~$\Lval, \Lda, \Ldd$). Its cardinality is denoted, by our convention, as $\lpk$ (resp.~$\lval, \lda, \ldd$). For $\pi \in \SS_n$, let $\Phi_{\FV}(\pi)=(w,h,c)$, where
\begin{equation*}
w_i= \left\{
  \begin{array}{ll}
  \S, & \mbox{  $ i \in \Lpk(\pi)$,} \\[3pt]
  \N, & \mbox{ $ i \in \Lval(\pi)$,}\\[3pt]
  \E, & \mbox{  $ i \in \Lda(\pi)$,}\\[3pt]
  \dE, & \mbox{  $ i \in \Ldd(\pi)$.}
 \end{array} \right.
\end{equation*}
and $c_{i}=\ppwrs_{\pi^{-1}(i)}(\pi)+\chi(w_{i}=\SdE)$. Here $\chi(S)=1$ if the statement $S$ is true and otherwise it equals $0$.

Conversely, for $W=(w,h,c)\in\LH_n$, we sketch the procedure to recover its preimage $\pi$. Beginning with $\pi^{(0)}=\emptyset$, we proceed to insert $i$ into $\pi^{(i-1)}$ to form $\pi^{(i)}$, for $i=1,2,\ldots,n$, and we set $\pi=\pi^{(n)}$ in the end. Now if $w_i=\S$ (resp.~$\N$, $\E$, $\dE$), then we insert $i$ (resp.~$\sqcup i \sqcup$, $i\sqcup$, $\sqcup i$) into the $c_i$-th empty slot in $\pi^{(i-1)}$ (here the counting is from right to left and starts with $0$), where $\sqcup$ represents an empty slot that will be filled by larger numbers later (with a possible exception with the rightmost $\sqcup$, which may stay being empty and thus corresponds to $\pi(n+1)=\infty$). We leave it to the reader to verify that indeed $\Phi_{\FV}(\pi)=W$, and therefore $\Phi_{\FV}$ is a bijection. As an example, if $\pi=618742593$, then $\Phi_{\FV}(\pi)$ is exactly the sr-Laguerre history given in Figure~\ref{fig:reslagu}.

Now we are ready for the main results of this subsection.

\begin{proposition}\label{prop:FV}
The bijection $\Phi_{\FV}:\SS_n\rightarrow \LH_n$ as given above links nine permutation statistics with their counterparts over sr-Laguerre histories as follows. For any $\pi\in\SS_n$, let $W=\Phi_{\FV}(\pi)$, then we have $\pi(n)=\cs(W)$, and
\begin{align}
\label{eq:FVSET}
& (\Dtb,\Dta,\Dbb,\Dba,\Abb,\Aba,\Id,\Ddif,\Dt\sqcup\rightdescent )\,\pi \\[3pt]
& \quad = (\Sdeb,\Sdea,\Ndeb,\Ndea,\Neb,\Nea, \Asc,\Het,\Wt)\,W,  \nonumber \\
\label{eq:FVSETcoro}
& (\rightascent, \rightdescent, \leftdescent)\,\pi =( \Wt\setminus\Nea\setminus\Sdea,
\Wt\setminus\Sdeb\setminus\Sdea,\Het\setminus\Wt)\,W.
\end{align}
\end{proposition}

\begin{proof}
Equality \eqref{eq:FVSET} should be clear from the definitions of the statistics and the construction of $\Phi_{\FV}$, then equality \eqref{eq:FVSETcoro} follows from \eqref{eq:FVSET} and the facts that for any $\pi \in \SS_n$,
\begin{align*}
  \rightascent(\pi) & = \rightdescent(\pi) \setminus \Aba(\pi) \sqcup \Dtb(\pi),\\[3pt]
  \leftdescent(\pi) & = \Ddif(\pi) \setminus \Dt(\pi) \setminus \rightdescent(\pi).
\end{align*}
\end{proof}

Combining Proposition~\ref{prop:FV} with Corollary~\ref{coro:xi-stats}, we get immediately the following corollary.

\begin{corollary}\label{coro:desbased}
$\phi=\Phi_{\FV}^{-1}\circ \xi \circ \Phi_{\FV}:\SS_n\to\SS_n$ is a bijection such that for any $\pi \in \SS_n$,
\begin{align*}
  (\Dtb,\Dta,\Abb,\Aba,\rightascent,\rightdescent,\leftdescent) \, \pi
&=\kappa_{n+1}\circ(\Aba,\Abb,\Dta,\Dtb,\rightdescent,\rightascent,\leftdescent) \, \phi(\pi),\\[4pt]
[n-1] \setminus\,  \Db(\pi) &=  \kappa_n\circ\Db(\phi(\pi)),\\[4pt]
  [n-1] \setminus\,  \Id(\pi) &=  \kappa_n\circ\Id(\phi(\pi)).
 % \Mad(\pi)&=n+1-\Mad \cup \Dtb \cup \Abb \setminus \Aba \setminus \Dta \: (\phi (\pi)),\\[4pt]
%  \Madl(\pi)&=n+1-\Madl  \cup \Abb \setminus \Dta \: (\phi (\pi)).
\end{align*}
\end{corollary}

Noting that $|\Dtb(\pi)\sqcup\Dta(\pi)|=\des(\pi)$, we deduce the following numerical version of equidistribution by taking cardinalities of those multiset-valued statistics in Corollary~\ref{coro:desbased}.
\begin{corollary}
The quintuples $(\pp31-2,\pprs,\ppwrs,\des,\ides)$ and $(\pp31-2,\ppwrs,\pprs,n-1-\des,\linebreak$
$n-1-\ides)$ are equidistributed over $\SS_n$, that is,
\begin{align}
\label{eq:fivevari}
\sum_{\pi \in \SS_n} t^{\des(\pi)} s^{\ides(\pi)}& p^{\pprs(\pi)} q^{\ppwrs(\pi)} r^{\pp31-2(\pi)}=\sum_{\pi \in \SS_n} t^{n-1-\des(\pi)} s^{n-1-\ides(\pi)} p^{\ppwrs(\pi)} q^{\pprs(\pi)} r^{\pp31-2(\pi)}.
\end{align}
\end{corollary}
\begin{remark}
Note that taking the complement $$\pi\mapsto\pi^{c}:=(n+1-\pi(1))(n+1-\pi(2))\cdots(n+1-\pi(n))$$ readily yields the equidistribution between $(\pprs,\ppwrs,\des,\ides)$ and $(\ppwrs,\pprs,n-1-\des,\linebreak$
$n-1-\ides)$, but $\pp31-2(\pi)=\pp31-2(\pi^c)$ is not true in general. It might be interesting to look for a direct bijection not involving Laguerre histories that proves \eqref{eq:fivevari}.
\end{remark}

Setting $s = p = 1$ in \eqref{eq:fivevari}, we have
\begin{equation}\label{eq:threevari1}
  \sum_{\pi \in \SS_n} t^{\des(\pi)}  q^{\ppwrs(\pi)} r^{\pp31-2(\pi)} =
\sum_{\pi \in \SS_n} t^{n-1-\des(\pi)}  q^{\pprs(\pi)} r^{\pp31-2(\pi)}.
\end{equation}
Shin and Zeng \cite[Theorem~2]{SZ12} derived an expansion for $\sum_{\pi \in \SS_n} t^{\des(\pi)}  q^{\pprs(\pi)} r^{\pp31-2(\pi)}$ that implies in particular
\begin{equation}\label{eq:threevari2}
\sum_{\pi \in \SS_n} t^{\des(\pi)}  q^{\pprs(\pi)} r^{\pp31-2(\pi)} = \sum_{\pi \in \SS_n} t^{n-1-\des(\pi)}  q^{\pprs(\pi)} r^{\pp31-2(\pi)}.
\end{equation}
Combining (\ref{eq:threevari1}) and  (\ref{eq:threevari2}),
we obtain an alternative proof of
\begin{equation}\label{eq:threevari3}
  \sum_{\pi \in \SS_n} t^{\des(\pi)}  q^{\pprs(\pi)} r^{\pp31-2(\pi)} =
\sum_{\pi \in \SS_n} t^{\des(\pi)}  q^{\ppwrs(\pi)} r^{\pp31-2(\pi)},
\end{equation}
which was also shown by Shin and Zeng \cite[Eqn.~(39)]{SZ12} via the common continued fraction expansion possessed by both sides of the equation.

Notice that $\SS_n(\pp31-2)=\SS_n(312)$. Setting $r=0$ in (\ref{eq:threevari3}),  we arrive at
\begin{equation}\label{eq:twovari1}
  \sum_{\pi \in \SS_n(312)} t^{\des(\pi)}  q^{\pprs(\pi)}=
\sum_{\pi \in \SS_n(312)} t^{\des(\pi)}  q^{\ppwrs(\pi)}.
\end{equation}
This fact was  first explicitly stated by Fu, Tang, Han and Zeng \cite[Thm.~1.1]{FTHZ}, where another eight interpretations for the polynomial in \eqref{eq:twovari1} were found.

In the end of this subsection, we extend \eqref{eq:threevari2} to a multiset-valued version and prove it bijectively via Br\"and\'en's modified Foata--Strehl action (MFS-action); see \cite{FSt,Bra08,SWG}. Our proof of the next theorem involves the linear statistics $\lpk$, $\lval$, $\ldd$, and $\lda$, with new boundary condition $\pi(0)=\pi(n+1)=0$. To emphasize this subtle distinction, we write them as $\lpk^*$, $\lval^*$, $\ldd^*$, and $\lda^*$. The corresponding set-valed statistics $\Lpk^*$, $\Lval^*$, $\Ldd^*$, and $\Lda^*$ are understood similarly if needed.
\begin{theorem}\label{thm:tristat}
There is an involution
$\varphi: \SS_n \to \SS_n$  such that
\begin{align}
\label{eq:tristat}
  (\des,\rightascent,\leftdescent)\pi &=(n-1-\des,\rightascent,\leftdescent) \varphi(\pi),\\
  \rightdescent(\varphi(\pi)) &= \rightdescent(\pi)\sqcup \Ldd^*(\pi)\setminus\Lda^*(\pi).
\end{align}
In particular, taking cardinalities of the set-valued statistics in \eqref{eq:tristat}, we recover \eqref{eq:threevari2}.
\end{theorem}

The MFS-action, sometimes referred to as ``valley hopping'' (see for example \cite[Sect.~4.1]{Ath}), has become one of the standard tools for proving and combinatorially interpreting the positivity of $\gamma$-coefficients for the Eulerian polynomials and their variants. To make this paper self-contained, we include here its definition and one example.

\begin{Def}[MFS-action]%[Modified Foata--Strehl action]
Let $\pi\in\SS_n$ with boundary condition $\pi(0)=\pi(n+1)=0$,
 for  any  $x\in[n]$, the {\em$x$-factorization} of $\pi$ reads $\pi=w_1 w_2x w_3 w_4,$ where $w_2$ (resp.~$w_3$) is the maximal contiguous subword immediately to the left (resp.~right) of $x$ whose letters are all larger than $x$.  Following Foata and Strehl~\cite{FSt} we define the action $\varphi_x$ by
$$
\varphi_x(\pi)=w_1 w_3x w_2 w_4.
$$
Note that if $x$ is a linear double ascent (resp. linear double descent) of $\pi$, then $w_2=\varnothing$ (resp. $w_3=\varnothing$), and if $x$ is a linear peak then
$w_2=w_3=\varnothing$.
For instance, if $x=3$ and $\pi=28531746\in\SS_7$, then $w_1=2,w_2=85,w_3=\varnothing$ and $w_4=1746$.
Thus $\varphi_x(\pi)=23851746$. Br\"and\'en~\cite{Bra08} modified the map $\varphi_x$ to be
\begin{align*}
\varphi'_x(\pi):=
\begin{cases}
\varphi_x(\pi),&\text{if $x$ is not a linear valley of $\pi$};\\
\pi,& \text{if $x$ is a linear valley of $\pi$.}
\end{cases}
\end{align*}
It is clear that $\varphi'_x$'s are involutions and commute. For any subset $S\subseteq[n]$ we can then define the map $\varphi'_S :\SS_n\to \SS_n$ by
\begin{align*}
\varphi'_S(\pi)=\left(\prod_{x\in S}\varphi'_x\right)(\pi),
\end{align*}
which is again an involution on $\SS_n$. Hence the group $\Z_2^n$ acts on $\SS_n$ via the functions $\varphi'_S$, $S\subseteq[n]$. This action will be called  the {\em Modified Foata--Strehl action} ({\em MFS-action} for short).
\end{Def}

\begin{figure}[htb]
\setlength {\unitlength} {0.8mm}
\begin {picture} (90,40) \setlength {\unitlength} {1mm}
\thinlines
\put(24,8){\dashline{1}(1,0)(25,0)}%3
\put(50,8){\vector(1,0){0.1}}
\put(-22,14){\dashline{1}(-1,0)(32,0)}%5
\put(10,14){\vector(1,0){0.1}}
\put(6,18){\dashline{1}(-1,0)(-25,0)}%6
\put(-20,18){\vector(-1,0){0.1}}
\put(46,12){\dashline{1}(-1,0)(-18,0)}%4
\put(28,12){\vector(-1,0){0.1}}
\put(-37, -2){$0$}
\put(-35, 2){\circle*{1.3}}
\put(-7,31){\line(-1,-1){28}}
\put(-7,31){\circle*{1.3}}
\put(20,4){\line(-1,1){27}}
\put(-7,32){$9$}
\put(-24,14){\circle*{1.3}}\put(-23,10){$5$}
\put(6,18){\circle*{1.3}}\put(9,16){$6$}
\put(20,4){\circle*{1.3}}
\put(20,4){\circle*{1.3}}\put(19.1,0){$1$}
\put(24,8){\circle*{1.3}}\put(24.5,4.5){$3$}
\put(52,6){\circle*{1.3}}\put(51.2,2){$2$}
\put(20,4){\line(1,1){17}}\put(37,21){\circle*{1.3}}
\put(37,21){\line(1,-1){15}}\put(36.5,22){$7$}
\put(46,12){\circle*{1.3}}\put(48,10){$4$}
\put(52,6){\line(1,1){21}}\put(73,27){\circle*{1.3}}\put(72.5,28.5){$8$}
\put(72.5,27){\line(1,-1){25}}
\put(97, 2){\circle*{1.3}}
\put(97,-2){$0$}
\end{picture}
\caption{MFS-actions on $596137428$ (recall $\pi(0)=\pi(10)=0$)
\label{valhop}}
\end {figure}
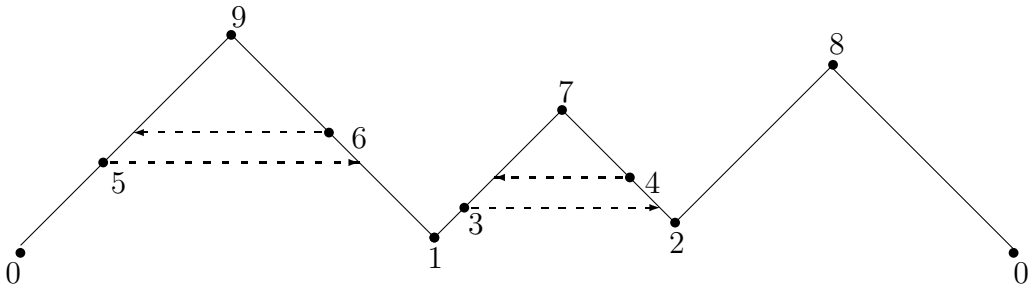

See Figure~\ref{valhop} for an illustration, where  exchanging
$w_2$ and  $w_3$ in the $x$-factorization is equivalent to
moving $x$ from a linear double ascent to a linear double descent or vice versa.

% Note that the boundary condition does matter. Take the permutation $596137428$ in Figure~\ref{valhop} as an example.
% %If $\pi(0)=0$, then $5$ is a double ascent and will be moved by $\varphi'_5$.
% If $\pi(0)=10$ instead, then $5$ becomes a valley and will be fixed by $\varphi'_5$.

Now the map we need for Theorem~\ref{thm:tristat} is precisely $\varphi:=\varphi'_{[n]}$. In other words, the points always hop whenever they can. So we see $\varphi(596137428)=695147328$ for the permutation in Figure~\ref{valhop}.

% Given a word $w$ of $n$ different integers,
% let $w=xmy$ with $m$ is the smallest element of $w$, where $x$ and $y$
% may be empty. We define $\varphi(w)$ by
% \begin{equation*}
% \varphi(w)= \left\{
%   \begin{array}{ll}
%   \varphi(y)m, & \mbox{  $x$ is empty and $y$ is not empty,} \\[3pt]
%   m\varphi(x), & \mbox{ $x$ is not empty and $y$ is  empty,}\\[3pt]
%   \varphi(x)m\varphi(y), & \mbox{  neither $x$ nor $y$ is empty.}
%  \end{array} \right.
% \end{equation*}

% As an example, let $\pi=25314876$ and $\varphi(\pi)$ can be
% computed as follows
% \begin{align*}
%   \varphi(\pi) &=\varphi(253)1\varphi(4876)  \\[3pt]
%    & =\varphi(53)21\varphi(876)4\\[3pt]
%    &=3\varphi(5)216\varphi(87)4\\[3pt]
%    &=352167\varphi(8)4\\[3pt]
%    &=35216784.
% \end{align*}

% Now, we are ready to give a proof of  Lemma \ref{lem:tristat}.
\begin{proof}[Proof of Theorem~\ref{thm:tristat}]
We have seen that $\varphi$ as defined above is an involution. It remains to verify \eqref{eq:tristat}. The definition of $\varphi$ yields the following facts for each $\pi\in\SS_n$:
\begin{align*}
& \lpk^*(\pi)=\lval^*(\pi)+1,\\
& \des(\pi)=\lval^*(\pi)+\ldd^*(\pi),\\
& \ldd^*(\pi)=\lda^*(\varphi(\pi)),~\lda^*(\pi)=\ldd^*(\varphi(\pi)),\\
& \lpk^*(\pi)=\lpk^*(\varphi(\pi)),~\lval^*(\pi)=\lval^*(\varphi(\pi)).
\end{align*}
Then we can compute
\begin{align*}
\des(\pi) &= \lval^*(\pi)+\ldd^*(\pi) = n-\lpk^*(\pi)-\lda^*(\pi) = n-\lpk^*(\varphi(\pi))-\ldd^*(\varphi(\pi))\\
&= n-(\lval^*(\varphi(\pi))+1)-\ldd^*(\varphi(\pi)) = n-1-\des(\varphi(\pi)).
\end{align*}
For the remaining three equalities, we make a key observation.
\begin{Fact}\label{fact:pkval}
For $1\le i\le n$, $\pprs_i(\pi)$ is the number of pairs $(\pi(j),\pi(k))\in\Lval^*(\pi)\times\Lpk^*(\pi)$ such that 1) $i<j<k$; 2) $\pi(j)<\pi(i)<\pi(k)$; and 3) $\pi(j)$ and $\pi(k)$ are consecutive (linear) valley and peak, in the sense that for each $j<l<k$, $\pi(l)$ is neither a valley nor a peak. Similarly, $\ppwrs_i(\pi)$ is the number of pairs $(\pi(j),\pi(k))\in\Lpk^*(\pi)\times\Lval^*(\pi)$ such that 1) $i<j<k$; 2) $\pi(k)<\pi(i)<\pi(j)$; and 3) $\pi(j)$ and $\pi(k)$ are consecutive peak and valley. $\pp31-2_i(\pi)$ is the number of pairs $(\pi(j),\pi(k))\in\Lpk^*(\pi)\times\Lval^*(\pi)$ such that 1) $j<k<i$; 2) $\pi(k)<\pi(i)<\pi(j)$; and 3) $\pi(j)$ and $\pi(k)$ are consecutive peak and valley.
\end{Fact}
Recall that the map $\varphi$ preserves all the peaks and valleys, and when a point hops, it passes by one peak. Combining this with Fact~\ref{fact:pkval}, we deduce that for $1\le i\le n$,
\begin{align*}
\pprs_i(\pi) &=\pprs_i(\varphi(\pi)),\\
\pp31-2_i(\pi) &=\pp31-2_i(\varphi(\pi)),\\
\ppwrs_i(\pi) &=\ppwrs_i(\varphi(\pi))+\chi(\pi(i)\in\Lda^*(\pi))-\chi(\pi(i)\in\Ldd^*(\pi)).
\end{align*}
Turning them into the results for the corresponding multiset-valued statistics completes the proof.
\end{proof}

%%%%%%%%%%%%%%%%%%%%%%%%%%%%%%%%%%%
\subsection{Application with \texorpdfstring{$\Phi_{\FZ}$}{Foata-Zeilberger's bijection}}\label{subsec:FZ}

Foata and Zeilberger gave another bijection
from $\SS_n$ to $\mathtt{L}_n^*$ in \cite{FZ90}, where $\mathtt{L}_n^*$ was referred to as the set of ``weighted paths''. Here we will present a shifted variant $\Phi_{\FZ}$ from $\SS_n$ to $\LH_n$.
Before that, we recall some cyclic permutation statistics. For a permutation
$\pi=\pi(1)\pi(2)\ldots \pi(n) \in \SS_n$, set
\begin{align*}
   &\Exc(\pi) =\{\pi(i) : \pi(i) > i, 1 \leq i \leq n\},  \\[3pt]
  &\Nexc(\pi) =\{\pi(i) : \pi(i) \leq i, 1 \leq i \leq n\},  \\[3pt]
   &\Excp(\pi)  =\{i: \pi(i) > i, 1 \leq i \leq n\}
\end{align*}
to be the set of {\em excedances, non-excedances} and {\em excedance positions} of $\pi$. We define the corresponding refined statistics as
\begin{align*}
  \Excb(\pi)=\{\pi(i) : \pi(i) <\pi(n), i \in \Excp(\pi)\}, &\,\,\,\, \Exca(\pi)=\{\pi(i) : \pi(i) >\pi(n), i \in \Excp(\pi)\}, \\[3pt]
  \Nexcb(\pi)=\{\pi(i) : \pi(i)<\pi(n), i \notin \Excp(\pi)\}, & \,\,\,\,\Nexca(\pi)=\{\pi(i) : \pi(i) >\pi(n), i \notin \Excp(\pi)\}, \\[3pt]
  \Excpb(\pi)=\{i : i <\pi(n), i \in \Excp(\pi)\}, &\,\,\,\, \Excpa(\pi)=\{i : i >\pi(n), i \in \Excp(\pi)\}.
\end{align*}
Further, we need the multiset-valued statistics {\em excedance difference} and {\em excedance position sum}\footnote{Here again, $\Excp(\pi)$ completely determines $\Ebot(\pi)$ and vice versa.} defined as
\begin{align*}
\Edif(\pi) &=\bigcup_{i \in \Excp(\pi)}\{i+1,i+2, \cdots, \pi(i)\},\\
\Ebot(\pi) &=\bigcup_{i \in \Excp(\pi)}\{i^{i}\}.
\end{align*}

The {\em excedance subword} of $\pi$, denoted by $\pi_E$,
is the word consisting of $\pi(i)$ with $i \in \Excp(\pi)$ in the order
induced by $\pi$. The {\em non-excedance subword} of $\pi$ is denoted
by $\pi_N$ and consists of $\pi(i)$ with $i \notin \Excp(\pi)$.
Given $i \in \Excp(\pi)$, we say the {\em imversion bottom number}
of $\pi(i)$ is $d$, if there are  exactly $d$ letters in $\pi_E$ to the left of $\pi(i)$ that are greater than $\pi(i)$. Similarly, for $i \notin \Excp(\pi)$, we say the {\em inversion top number} of $\pi(i)$ is $d$, if there are exactly $d$ letters in $\pi_N$ to the
right of $\pi(i)$ that are smaller than $\pi(i)$. The {\em side number} of $i$, denoted by $s_i$, is either the imversion bottom number
of $\pi(i)$ if $i \in \Excp(\pi)$, or the inversion top number of $\pi(i)$ if $i \notin \Excp(\pi)$. And the sequence $s=s_1 s_2 \cdots s_n$ is referred to as the side number of $\pi$. We define $\Ine(\pi)$ to be the set containing $s_i$ copies of $\pi(i)$.

\begin{example}
For $\pi=947612853$, we have $\pi_E=9 4 7 6 8$ and $\pi_N=1 2 5 3$. Moreover, we see that $\Exc(\pi)=\Exca(\pi)=\{4,6,7,8,9\}$, $\Excb(\pi)=\emptyset$,
$\Excp(\pi)=\{1,2,3,4,7\}$, $\Excpb(\pi)=\{1,2\}$,
$\Excpa(\pi)=\{4,7\}$,
$\Nexc(\pi)=\{1,2,3,5\}$,
$\Nexcb(\pi)=\{1,2\}$,
$\Nexca(\pi)=\{5\}$,
$\Edif(\pi)=\{2,3^2,4^3,5^3,6^3,7^2,8^2,9\}$, and
$\Ebot(\pi)=\{1,2^2,3^3,4^4,7^7\}$. The side number of $\pi$ is $011200110$, thus $\Ine(\pi)=\{4,5,6^2,7,8\}$.
\end{example}

Now we are ready to rephrase the bijection $\Phi_{\FZ}$ of
Foata and Zeilberger as follows. Given $\pi=\pi(1) \pi(2) \cdots \pi(n) \in \SS_n$,  $i$ is called a
\begin{itemize}
  \item {\em cyclic double ascent} if $\pi^{-1}(i) < i <\pi(i)$;
  \item {\em cyclic double descent} if $\pi^{-1}(i) \geq  i \geq \pi(i)$;
  \item {\em cyclic peak} if $\pi^{-1}(i) < i >\pi(i)$;
  \item {\em cyclic valley} if $\pi^{-1}(i) > i <\pi(i)$.
\end{itemize}
The set of cyclic peaks (resp.~valleys, double ascents, double descents) is denoted by $\Cpk$ (resp.~$\Cval, \Cda, \Cdd$).
%The corresponding cardinality are defined by $\cpk$ (resp. $\cval, \cda, \cdd$), respectively.

For $\pi \in \SS_n$ whose side number is given by $s=s_1 s_2 \cdots s_n$, we define $\Phi_{\FZ}(\pi)=(w,h,c)$, where for $1\le i\le n$, $c_{i}=s_{\pi^{-1}(i)}+\chi(w_{i}=\SdE)$,
%$c_{\pi(i)}=s_i+\chi(w_{\pi(i)}=SdE)$
\begin{equation*}
w_i= \left\{
  \begin{array}{ll}
\N, & \mbox{ $ i \in \Cval(\pi)$,}\\[3pt]
\S, & \mbox{  $ i \in \Cpk(\pi)$,} \\[3pt]
\E, & \mbox{  $ i \in \Cdd(\pi)$,}\\[3pt]
\dE, & \mbox{  $ i \in \Cda(\pi)$,}
 \end{array} \right.
\end{equation*}
and $h$ is uniquely deduced from $w$.

Conversely, given $W=(w,h,c) \in \LH_n$, we may recover $\pi=\Phi_{\FZ}^{-1}(W)$ as follows. Firstly, we construct two words, which we call $w_E$ and $w_N$, in two-line notations.
The first row of $w_E$ (resp. $w_N$) is monotonously increasing, and it consists of all $i$ with $w_i=\NdE$ (resp. $w_i=\SE$), while the second row of $w_E$ (resp. $w_N$) is composed of all $i$ with $w_i=\SdE$ (resp. $w_i=\NE$). The elements in the second row of $w_E$ (resp. $w_N$) are ordered so that the imversion bottom number (resp.~inversion top number) of the element $i$ is $c_{i}-\chi(w_{i}=\SdE)$. To be specific, the arrangement of the elements in the second row of $w_E$ is given as follows. Pick out the smallest element in the second row of $w_E$, say $x$, and place it at the position such that there are $c_x-1$ empty slots to the left of it. Then, take the smallest element in what remains in the second row of $w_E$, say $y$, and place it at the position such that there are $c_y-1$ empty slots to the left of it. Repeating like this until
all the elements in the second row of $w_E$ have been placed. For the arrangement of the elements in the second row of $w_N$, we proceed analogously as follows. Pick out the largest element in the second row of $w_N$, say $x$, and place it at the position such that there are $c_x$ empty slots to the right of it. Then, pick out the largest remaining element, say $y$, and place it at the position such that there are $c_y$ empty slots to the right of it. Repeating like this until
all the elements in the second row of $w_N$ are in position. After that, we concatenate $w_E$ and $w_N$, then rearrange the columns so that the first row becomes increasing. What we have obtained in the second row of the final output is taken to be the preimage $\pi$.

As an example, for $\pi=947612853$,  $\Phi_{\FZ}(\pi)$ is exactly the sr-Laguerre history given in Figure
\ref{fig:reslagu}. For the sr-Laguerre history $W$ given in Figure
\ref{fig:reslagu}, we deduce that
\begin{equation*}
w_E= \left(
  \begin{array}{ll}
1\,\,\,  2 \,\,\, 3 \,\,\, 4 \,\,\, 7\\[3pt]
9 \,\,\, 4 \,\,\, 7 \,\,\, 6 \,\,\, 8
 \end{array} \right)  \,\,\,\,\,\, \text{and} \,\,\,\,\,\,
 w_N= \left(
  \begin{array}{ll}
5\,\,\,  6 \,\,\, 8 \,\,\, 9\\[3pt]
1 \,\,\, 2 \,\,\, 5 \,\,\, 3
 \end{array} \right).
\end{equation*}
Finally, we get $\pi=\Phi_{\FZ}^{-1}(W)=947612853$.

Let $|S|_i$ be the number of times $i$ occurs in the multiset $S$. Given integers $a$ and $b$, we denote by $(a,b]$ the set $\{a+1,a+2,\ldots,b\}$. Note that when $a\ge b$, $(a,b]$ is simply the empty set. The following proposition parallels Proposition~\ref{prop:FV}.

\begin{proposition}\label{prop:FZ}
The bijection $\Phi_{\FZ}:\SS_n\rightarrow \LH_n$ links twelve permutation statistics with their counterparts over sr-Laguerre histories as follows. For any $\pi\in\SS_n$ and $W=\Phi_{\FZ}(\pi)$, we have $\pi(n)=\cs(W)$,
\begin{align}
\label{eq:FZSET}
(\Excb,\Exca,\Excpb,\Excpa,\Nexcb,\Nexca,\Edif,\Exc\sqcup\Ine)\:\pi \\[3pt]
= (\Sdeb,\Sdea,\Ndeb,\Ndea,\Neb,\Nea, \Het,\Wt)\:W,  \nonumber
\end{align}
and
\begin{align}
\label{eq:FZSET2}
&(\Ine \sqcup \Excb \setminus \Nexca, \Ine, \Edif \setminus \Exc \setminus \Ine)\:\pi \\
&= ( \Wt\setminus\Nea\setminus\Sdea,
\Wt\setminus\Sdeb\setminus\Sdea,\Het\setminus\Wt)\:W. \nonumber
\end{align}
\end{proposition}

\begin{proof}
Firstly, we wish to prove that $\pi(n)=\cs(W)$.
Since $\pi(n) \leq n$, we see $\pi(n)\in \Cval(\pi)$ or $\pi(n)\in \Cdd(\pi)$. It follows that $w_{\pi(n)}=\NE$. On the other hand, we clearly have $s_n=0$, hence the label $c_{\pi(n)}=0$. It remains to show that $c_j>0$ for every $j>\pi(n)$. Recall that
$$c_j=s_{\pi^{-1}(j)}+\chi(w_j=\SdE).$$
We consider two cases. If $j > \pi^{-1}(j)$, then $w_j=\SdE$, which implies that $c_j>0$. Otherwise $j \leq \pi^{-1}(j)$, then from the fact that $j>\pi(n)$ we deduce that the inversion top number of $j$, which in this case equals $s_{\pi^{-1}(j)}$, is at least $1$. Hence, we also have $c_j>0$. We see that $c_{\pi(n)}$ is indeed the last occurrence of $0$, thus $\pi(n)=\cs(W)$ by the definition of the critical step.

Next, we proceed to prove (\ref{eq:FZSET}).
We present here an inductive proof of the multiset-valued equality $\Edif(\pi)=\Het(W)$, while the others can be verified directly through the definitions of the statistics and the construction of the bijection $\Phi_{\FZ}$.

Clearly, neither $\Edif(\pi)$ nor $\Het(W)$ contains $1$. Now suppose that $i$ occurs as many times in $\Edif(\pi)$ as in $\Het(W)$, we wish to show that the same holds for $i+1$. There are three cases to consider:
\begin{itemize}
    \item If $h_{i+1}=h_{i}+1$, then $w_{i}=\N$ and $|\Het(W)|_{i+1}=|\Het(W)|_{i}+1$. Using the induction hypothesis, it suffices to prove that $|\Edif(\pi)|_{i+1}=|\Edif(\pi)|_{i}+1$. Clearly, we have $i \in \Cval(\pi)$, namely, $\pi^{-1}(i)>i<\pi(i)$. This implies that $i+1 \in (i,\pi(i)]$, $(\pi^{-1}(i),i]=\emptyset$, and for $1 \leq j < i$, the interval $(j,\pi(j)]$ contains $i$ if and only if it contains $i+1$, as desired.
    \item If $h_{i+1}=h_{i}$, then $w_{i}=\E$ or $w_{i}=\dE$, and $|\Het(W)|_{i+1}=|\Het(W)|_{i}$. We need to show that $|\Edif(\pi)|_{i+1}=|\Edif(\pi)|_{i}$. If $w_{i}=\E$, then $\pi^{-1}(i) \geq i \geq \pi(i)$. It indicates that $(i,\pi(i)]=(\pi^{-1}(i),i]=\emptyset$, and for $1 \leq j < i$, the interval $(j,\pi(j)]$ contains $i$ if and only if it contains $i+1$, as desired. If $w_{i}=\dE$, then $\pi^{-1}(i)<i<\pi(i)$. It follows that $i$ and $i+1$ each occurs once in the union $(i,\pi(i)]\cup (\pi^{-1}(i),i]$. Moreover, it is easy to check that for $1 \leq j < i$ and $j \neq \pi^{-1}(i)$, the interval $(j,\pi(j)]$ contains $i$ if and only if it contains $i+1$. So in both cases, we have $|\Edif(\pi)|_{i+1}=|\Edif(\pi)|_{i}$.
    \item If $h_{i+1}=h_{i}-1$, then $w_{i}=\S$ and $|\Het(W)|_{i+1}=|\Het(W)|_{i}-1$. We wish to show that $|\Edif(\pi)|_{i+1}=|\Edif(\pi)|_{i}-1$. Clearly, we have $i \in \Cpk(\pi)$, namely, $\pi^{-1}(i)<i>\pi(i)$. We observe analogously as in the previous two cases that $i\in (\pi^{-1}(i),i]$, $(i,\pi(i)]=\emptyset$, and for $1 \leq j < i$ and $j \neq \pi^{-1}(i)$, the interval $(j,\pi(j)]$ contains $i$ if and only if it contains $i+1$.
\end{itemize}

Combining all the cases above, we obtain that $\Edif(\pi)=\Het(W)$. Notice that (\ref{eq:FZSET2})  follows directly from (\ref{eq:FZSET}). The proof is now completed.
\end{proof}

By Propositions \ref{prop:FV} and \ref{prop:FZ}, we immediately get the following corollary, which can be viewed as a multiset-valued generalization of an equidistribution result for $\Phi_{\CSZ}:=\Phi_{\FZ}^{-1} \circ \Phi_{\FV}$ as given by Clarke, Steingr\'{i}msson and Zeng \cite{CSZ}. It should be noted here that the mapping $\Phi_{\CSZ}$ was originally defined via the two bijections of Fran\c con-Viennot and Foata-Zeilberger which are from $\SS_n$ to $\mathtt{L}_n^*$. Nontheless, it is easy to check that
the change of the image sets of both $\Phi_{\FV}$ and $\Phi_{\FZ}$ brings no difference in the composition $\Phi_{\FZ}^{-1} \circ \Phi_{\FV}$. In other words, $\Phi_{\CSZ}= \Phi_{\FZ}^{-1} \circ \Phi_{\FV}$ with our modified versions of $\Phi_{\FV}$ and $\Phi_{\FZ}$ indeed matches the mapping in \cite{CSZ}.

\begin{corollary}
Given $\pi \in \SS_n$, we have $\pi(n)=\Phi_{\CSZ}(\pi)(n)$, and
\begin{align*}
  (\Dt,\Db,\Ab,\rightascent, \rightdescent, \leftdescent,\Dbot,\Ddif)\:\pi&= \\[3pt]
(\Exc,\Excp,\Nexc,\Ine \sqcup \Excb \setminus &\Nexca, \Ine, \Edif \setminus \Exc \setminus \Ine,\Ebot,\Edif) \:\Phi_{\CSZ}(\pi).
\end{align*}
\end{corollary}

%\begin{proof}
%\end{proof}
 %We note that (\ref{eq:setmahonian}) is a set-valued generation of equally distributions of Mahonian statistics given in \cite{CSZ}.
As a parallel result of Corollary \ref{coro:desbased},
we have the following corollary.

\begin{corollary}
Let $\eta=\Phi_{\FZ}^{-1}\circ \xi \circ \Phi_{\FZ}$ and $\pi \in \SS_n$, then
\begin{align*}
  (\Exc,\Nexc,\Ine \sqcup \Excb& \setminus \Nexca ,\Ine,\Edif \setminus \Exc \setminus \Ine) \: \pi \\[4pt]
&=\kappa_{n+1}\circ(\Nexc,\Exc,\Ine,\Ine \sqcup \Excb \setminus \Nexca ,\Edif \setminus \Exc \setminus \Ine) \:\eta(\pi),\\[4pt]
&[n-1] \setminus\,  \Excp(\pi)=  \kappa_n\circ\Excp(\eta(\pi)).
 % \Inv(\pi)&=n+1-\Inv \cup \Excb \cup \Excpb \setminus \Nexca \setminus \Exca \: (\eta (\pi)),\\[4pt]
%  \Madl'(\pi)&=n+1-\Madl'  \cup \Nexcb \setminus \Exca \: (\eta (\pi)).
\end{align*}
\end{corollary}

%%%%%%%%%%%%%%%%%%%%%%%%%%%%%%%%%%%%%
\subsection{Application with \texorpdfstring{$\Phi_{\YZL}$}{Yan-Zhou-Lin's bijection}}\label{subsec:YZL}

Recently, Yan, Zhou and Lin \cite{YZL} defined a bijection
 $\Psi_{\YZL}$ from $\SS_{n+1}$ to $\mathtt{L}_n$, which can be seen as
 a shifted analogue of that of Foata and Zeilberger. For $\pi \in \SS_{n+1}$ and $1\le i\le n+1$, let
 \[\nest_i(\pi)= \# \{~j \colon j <i< \pi(i) < \pi(j)~\text{or}~ \pi(j) < \pi(i) \leq i <j\},\]
 and define
 \begin{align*}
   \Scval(\pi) &= \{i \in [n] \colon  i < \pi(i) ~ \text{and}~  i+1 \leq \pi^{-1}(i+1) \}; \\
   \Scpk(\pi) & =\{i \in [n] \colon i \geq  \pi(i) ~ \text{and}~  i+1 > \pi^{-1}(i+1)\};\\
   \Scda(\pi) & =\{i \in [n] \colon i < \pi(i) ~ \text{and}~  i+1 > \pi^{-1}(i+1)\};\\
   \Scdd(\pi) & =\{i \in [n] \colon i  \geq \pi(i) ~ \text{and}~  i+1 \leq \pi^{-1}(i+1)\}.
 \end{align*}
 Then $\Psi_{\YZL}(\pi)=(w,h,c)$ can be defined as $c_i=\nest_i(\pi)$
 and
 \begin{equation}
 \label{eq:step}
w_i= \left\{
  \begin{array}{ll}
\N, & \mbox{ $ i \in \Scval(\pi)$,}\\[3pt]
\S, & \mbox{  $ i \in \Scpk(\pi)$,} \\[3pt]
\E, & \mbox{  $ i \in \Scda(\pi)$,}\\[3pt]
\dE, & \mbox{  $ i \in \Scdd(\pi)$.}
 \end{array} \right.
\end{equation}

 Recall that the images of $\Psi_{\YZL}$
 are in $\mathtt{L}_n$. In order to make the images fall in $\LH_n$, modifications on the choice vector $c$ are required. Indeed, we proceed to define a shifted and restricted version $\Phi_{\YZL}:\SS_n\to\LH_n$, which is going to be composed with our
 involution $\xi$ on sr-Laguerre histories.

Given $\pi \in \SS_n$, let $\pone(\pi)=\pi^{-1}(1)$ be the position of $1$ in $\pi$.  We construct an sr-Laguerre history $\Phi_{\YZL}(\pi):=(\widetilde{w},\widetilde{h},\widetilde{c}) \in \LH_n$ as follows.

For $i \in [n-1]$, if $i \neq \pone(\pi)$, then
\begin{equation*}
\widetilde{w}_i= \left\{
  \begin{array}{ll}
\N, & \mbox{ $ i \in \Scval(\pi)$,}\\[3pt]
\S, & \mbox{  $ i \in \Scpk(\pi)$,} \\[3pt]
\E, & \mbox{  $ i \in \Scda(\pi)$,}\\[3pt]
\dE, & \mbox{  $ i \in \Scdd(\pi)$.}
 \end{array} \right.
\end{equation*}
 If $i=\pone(\pi)$, then
 \[\widetilde{w}_i=\begin{cases}
    \N, & ~~\text{if $i+1 \leq \pi^{-1}(i+1),$} \\[3pt]
	\E, & ~~\text{if $i+1 > \pi^{-1}(i+1).$}
	\end{cases}
\]
And let
\[\widetilde{w}_n=\begin{cases}
    \S, & ~~\text{if $\pone(\pi) \neq n,$} \\[3pt]
	\E, & ~~\text{if $\pone(\pi) = n.$}
	\end{cases}
\]
The height vector $\widetilde{h}$ is then completely determined by $\widetilde{w}$ as in (2) of Definition~\ref{def:LH}. For $i \in [n]$, let $\widetilde{c}_i=\vnest_i(\pi)+\chi(w_i=\SdE)$, where $\vnest_i(\pi)$ is a variant of $\nest_i(\pi)$ defined as
\begin{equation}\label{def:vnest}
\vnest_i(\pi)=\begin{cases}
	\nest_i(\pi)-1, & \text{if $\pi(i) \leq i$ and $i < \pone(\pi)$,} \\[3pt]
	\nest_i(\pi)+1, & \text{if $\pi(i) > i$ and $i > \pone(\pi)$,} \\[3pt]
	\nest_i(\pi), & \text{otherwise}.
	\end{cases}
\end{equation}

At this point, the reader is encouraged to use the permutation $\sigma=671395482$ in Figure~\ref{fig:perdiagram} as an example and work out its image $\Phi_{\YZL}(\sigma)$, which is precisely the sr-Laguerre history depicted in Figure~\ref{fig:reslagu}.

To show that $\Phi_{\YZL}$ is well defined and indeed a bijection, we need to utilize the notion of ``permutation diagram''; see \cite{Cor} for more
information. Given
$\sigma \in \SS_n$, its permutation diagram can be obtained as follows.
Firstly, draw a line which is marked with the numbers $1,2,\ldots, n$ from left to right.
Then, draw an arc from $i$ to $\sigma(i)$ above the line
if $i<\sigma(i)$ and under the line otherwise.
Note that the definition given here is slightly different from that in \cite{Cor}, as the arcs corresponding to the fixed points are placed under the line. Take $\sigma=671395482$ for example, its permutation diagram is drawn in Figure~\ref{fig:perdiagram}.
\begin{figure}[!htbp]
\begin{center}
\begin{tikzpicture}[line width=0.8pt,scale=0.65]
%\tikzset{every square node/.style={text=black, inner color=gray, outer color =white, inner sep=0.1cm, aspect=2}}
%\tikzset{mynode/.style=rectangle,draw top color=red, draw bottom color=white,line width=2pt}
\tikzstyle{every node}=[font=\small,scale=0.85]
\coordinate (O) at (0,0);

\draw[thick] (O) to ++(9,0);
\fill[black!100] (O)++(0.5,0) circle(0.5ex)++(1,0) circle(0.5ex)
++(1,0) circle(0.5ex)++(1,0) circle(0.5ex)++(1,0) circle(0.5ex)
++(1,0) circle(0.5ex)++(1,0) circle(0.5ex)++(1,0) circle(0.5ex)
++(1,0) circle(0.5ex);

\draw[thick] (O)++(0.5,0)  to [out=45,in=135] ++(5,0);
\draw[thick] (O)++(1.5,0)  to [out=45,in=135] ++(5,0);
\draw[thick] (O)++(2.5,0)  to [out=-130,in=-60] ++(-2,0);
\draw[thick] (O)++(3.5,0)  to [out=-130,in=-60] ++(-1,0);
\draw[thick] (O)++(4.5,0)  to [out=45,in=135] ++(4,0);
\draw[thick] (O)++(5.5,0)  to [out=-130,in=-60] ++(-1,0);
\draw[thick] (O)++(6.5,0)  to [out=-130,in=-60] ++(-3,0);
\draw[thick] (O)++(7.5,0)  to [out=-150,in=-150] ++(0,-0.4)
to [out=30,in=70] ++(0,0.4);
\draw[thick] (O)++(8.5,0)  to [out=-150,in=-40] ++(-7,0);

\path (O)++(0.5,0.35)  node {$1$} ++(1,0) node {$2$}
++(1,0) node {$3$}++(1,0) node {$4$}++(1,0) node {$5$}
++(1,0) node {$6$}++(1,0) node {$7$}++(1,0) node {$8$}
++(1,0) node {$9$};
\end{tikzpicture}
%%%%%%%%%%%%%%%%%%%%%%%%%%%%%%%%%%%%%%%%
\caption{The permutation diagram of $\sigma=671395482$.}
\label{fig:perdiagram}
\end{center}
\end{figure}
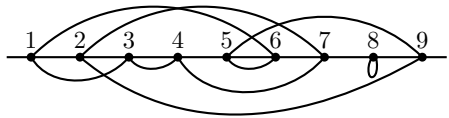

It can be easily checked that
for $i \neq \pone(\sigma)$ and $i \neq n$, $w_i$ equals $\N$, $\E$, $\dE$, and $\S$, if and only if $(i,\sigma(i))$ and $(\sigma^{-1}(i+1),i+1)$ are of the types (I), (II), (III), and (IV) given in Figure \ref{fig:fourtype}, respectively.
\begin{figure}[!htbp]
\begin{center}
\begin{tikzpicture}[line width=0.8pt,scale=0.6]
%\tikzset{every square node/.style={text=black, inner color=gray, outer color =white, inner sep=0.1cm, aspect=2}}
%\tikzset{mynode/.style=rectangle,draw top color=red, draw bottom color=white,line width=2pt}
\tikzstyle{every node}=[font=\small,scale=0.8]
\coordinate (O) at (0,0);
\coordinate (O1) at (4,0);
\coordinate (O2) at (8,0);
\coordinate (O3) at (12,0);

\draw[thick] (O) to ++(2,0);
\fill[black!100] (O)++(0.5,0) circle(0.5ex)++(1,0) circle(0.5ex);
\draw[ thick] (O)++(0.5,0)  to [out=70,in=180] ++(1,0.5);
\draw[ thick] (O)++(1.5,0)  to [out=-70,in=180] ++(0.5,-0.5);

\draw[thick] (O1) to ++(2,0);
\fill[black!100] (O1)++(0.5,0) circle(0.5ex)++(1,0) circle(0.5ex);
\draw[thick] (O1)++(0.5,0)  to [out=70,in=180] ++(0.3,0.5);
\draw[thick] (O1)++(1.5,0)  to [out=100,in=0] ++(-0.3,0.5);

\draw[thick] (O2) to ++(2,0);
\fill[black!100] (O2)++(0.5,0) circle(0.5ex)++(1,0) circle(0.5ex);
\draw[thick] (O2)++(1.5,0)  to [out=-80,in=180] ++(0.5,-0.5);
\draw[thick] (O2)++(0.5,0)  to [out=-70,in=0] ++(-0.5,-0.5);

\draw[thick] (O3) to ++(2,0);
\fill[black!100] (O3)++(0.5,0) circle(0.5ex)++(1,0) circle(0.5ex);
\draw[thick] (O3)++(1.5,0)  to [out=100,in=0] ++(-1,0.5);
\draw[thick] (O3)++(0.5,0)  to [out=-70,in=0] ++(-0.5,-0.5);

\path (O)++(1,-1)  node {(I)};
\path (O1)++(1,-1)  node {(II)};
\path (O2)++(1,-1)  node {(III)};
\path (O3)++(1,-1)  node {(IV)};

\end{tikzpicture}
%%%%%%%%%%%%%%%%%%%%%%%%%%%%%%%%%%%%%%%%
\caption{Types (I) to (IV) that $\N$, $\E$, $\dE$, $\S$ correspond to in a permutation diagram.}
\label{fig:fourtype}
\end{center}
\end{figure}
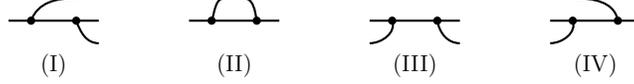

\begin{lemma}
The mapping $\Phi_{\YZL}$ is well defined. I.e., for each $\pi \in \SS_n$, suppose $W=(\widetilde{w},\widetilde{h},\widetilde{c})=\Phi_{\YZL}(\pi)$.
We have that $W \in \LH_n$.
\end{lemma}
\begin{proof}
Firstly, we show that for $1\le j\le n$,
\begin{align}
\label{eq:NS}
\#\{i: i < j,\:\widetilde{w}_i=\N\} &\geq \#\{i: i < j,\:\widetilde{w}_i=\S\},\text{ and}\\[3pt]
\#\{i: 1\le i\le n,\:\widetilde{w}_i=\N\} &= \#\{i: 1\le i\le n,\:\widetilde{w}_i=\S\}.
\label{eq:N=S}
\end{align}
Let $\bar{\pi}$ be a permutation of $[n+1]$ that is derived from $\pi$ by letting
\[
\bar{\pi}(i)=\begin{cases}
	n+1, & \text{if $i=\pone(\pi)$,} \\[3pt]
	1, & \text{if $i=n+1$,} \\[3pt]
	\pi(i), & \text{otherwise}.
	\end{cases}
\]
In other words, we insert $n+1$ to the immediate left of $1$ in the cycle notation of $\pi$ to get the cycle notation of $\bar{\pi}$. Suppose $\Psi_{\YZL}(\bar{\pi})=(w,h,c)$, then it can be routinely checked that $w=\widetilde{w}$ and thus $h=\widetilde{h}$. Notice that $\Psi_{\YZL}(\bar{\pi}) \in \mathtt{L}_n$, which implies \eqref{eq:NS} and \eqref{eq:N=S}, as well as that $0\le c_i\le h_i$, for $1\le i\le n$.

It remains to show that for every $1\le i\le n$, we have
\begin{align*}
\widetilde{h}_i\ge \widetilde{c}_i\ge \begin{cases}
  0, & \text{if } \widetilde{w}_i=\NE,\\
  1, & \text{if } \widetilde{w}_i=\SdE.
\end{cases}
\end{align*}
We confirm it according to the following cases.
\begin{itemize}
  \item  If $i < \pone(\pi)$ and $\pi(i) \leq i$, then $\widetilde{w}_i=\SdE$. Hence $\widetilde{c}_i=\vnest_i(\pi)+1=\nest_i(\pi)$ according to \eqref{def:vnest}. Now since $1 < \pi(i) \le i < \pone(\pi)$, we see that $\nest_i(\pi) \ge 1$. On the other hand, going from $\pi$ to $\bar{\pi}$, we see the nesting $1 < \pi(i) \le i < \pone(\pi)$ is replaced by the nesting $1 < \bar{\pi}(i) \le i < n+1$ while other nestings are preserved. Hence we have $\nest_i(\pi)=\nest_i(\bar{\pi})=c_i\le h_i=\widetilde{h}_i$, as desired.

  \item  If $\pone(\pi)<i<n$ and $\pi(i) \leq i$, then $\widetilde{w}_i=\SdE$. Hence $\widetilde{c}_i=\vnest_i(\pi)+1=\nest_i(\pi)+1$ according to \eqref{def:vnest}. This means that we need to show $1 \le \nest_i(\pi)+1 \le \widetilde{h}_i$. The first inequality is obvious. For the second one, we observe again that going from $\pi$ to $\bar{\pi}$, a new nesting $1 < \bar{\pi}(i) \le i < n+1$ is created while other nestings remain nested. Therefore $\nest_i(\pi)+1=\nest_i(\bar{\pi})=c_i\le h_i=\widetilde{h}_i$, as desired.

  \item  If $i < \pone(\pi)$ and $\pi(i) > i$, then $\widetilde{w}_i=\NE$. Hence $\widetilde{c}_i=\vnest_i(\pi)=\nest_i(\pi)$ according to \eqref{def:vnest}. Similar discussion with $\pi$ and $\bar{\pi}$ gives us $\nest_i(\pi)=\nest_i(\bar{\pi})$, thus $0\le \widetilde{c}_i=\nest_i(\bar{\pi})\le h_i=\widetilde{h}_i$, as desired.

  \item  If $\pone(\pi)<i<n$ and $\pi(i) > i$, then $\widetilde{w}_i=\NE$. Hence $\widetilde{c}_i=\vnest_i(\pi)=\nest_i(\pi)+1$ according to \eqref{def:vnest}. Now, the new nesting $\pone(\pi)<i<\bar{\pi}(i)<n+1$ indicates that $\nest_i(\bar{\pi})=\nest_i(\pi)+1$ and the desired inequalities follow.

  \item If $i=\pone(\pi) < n$, then $\nest_i(\pi)=0$ and $\widetilde{w}_i=\NE$. Hence $\widetilde{c}_i=\vnest_i(\pi)=\nest_i(\pi)=0$, which certainly agrees with $0 \le \widetilde{c}_i \le \widetilde{h}_i$.

  \item If $i=n$, then $\nest_i(\pi)=0$. We further consider two subcases. If $n \neq \pone(\pi)$, then $\widetilde{w}_i=\S$. It follows that $\widetilde{c}_i=\vnest_i(\pi)+1=\nest_i(\pi)+1=1$, which agrees with $1\le \widetilde{c}_i\le \widetilde{h}_i$. If $n=\pone(\pi)$, then $\widetilde{w}_i=\E$. It follows that $\widetilde{c}_i=\vnest_i(\pi)=\nest_i(\pi)=0$, which agrees with $0 \leq \widetilde{c}_i \le \widetilde{h}_i$.
\end{itemize}
\end{proof}

\begin{lemma}\label{lem:cspone}
Given $\pi \in \SS_n$ and $W=(w,h,c)=\Phi_{\YZL}(\pi)$,
we have $\pone(\pi)=\cs(W)$.
\end{lemma}

\begin{proof}
Suppose $\pone(\pi)=i$. By the definition of the critical step, it suffices to show that $c_i=0$ and $c_j\ge 1$ for all $j>i$. Now $w_i=\NE$ so $c_i=\vnest_i(\pi)=\nest_i(\pi)$ according to \eqref{def:vnest}, and $\nest_i(\pi)=0$ since $\pi(i)=1$. If we have in addition that $i=n$, then we are done. Otherwise we can assume $i<n$. In particular, this implies that $w_n=\S$.

Next, take any $j>i$, we consider two cases. If $w_j=\SdE$, then $c_j\ge 1$ holds. If $w_j=\NE$, then we must have $j<n$ since $w_n=\S$. Now we can rely on the correspondence in Figure~\ref{fig:fourtype} to deduce that $j < \pi(j)$. It follows that $c_j=\vnest_j(\pi)=\nest_j(\pi)+1 \ge 1$ as well.
\end{proof}

Now, we introduce in Figure \ref{fig:fourtype2} four types of semi-arcs in a permutation diagram. Among them, arcs of types ``AR'' and ``BL'' are called ``outward'', while those of types ``AL'' and ``BR'' are called ``inward''.

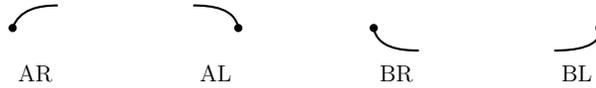
\begin{figure}[!htbp]
\begin{center}
\begin{tikzpicture}[line width=0.8pt,scale=0.6]
%\tikzset{every square node/.style={text=black, inner color=gray, outer color =white, inner sep=0.1cm, aspect=2}}
%\tikzset{mynode/.style=rectangle,draw top color=red, draw bottom color=white,line width=2pt}
\tikzstyle{every node}=[font=\small,scale=0.8]
\coordinate (O) at (0,0);
\coordinate (O1) at (4,0);
\coordinate (O2) at (8,0);
\coordinate (O3) at (12,0);

\fill[black!100] (O)++(0.5,0) circle(0.5ex);
\draw[ thick] (O)++(0.5,0)  to [out=70,in=180] ++(1,0.5);

\fill[black!100] (O1)++(1.5,0) circle(0.5ex);
\draw[thick] (O1)++(1.5,0)  to [out=100,in=0] ++(-1,0.5);

\fill[black!100] (O2)++(0.5,0) circle(0.5ex);
\draw[thick] (O2)++(0.5,0)  to [out=-80,in=180] ++(1,-0.5);

\fill[black!100] (O3)++(1.5,0) circle(0.5ex);
\draw[thick] (O3)++(1.5,0)  to [out=-80,in=0] ++(-1,-0.5);

\path (O)++(1,-1)  node {AR};
\path (O1)++(1,-1)  node {AL};
\path (O2)++(1,-1)  node {BR};
\path (O3)++(1,-1)  node {BL};

\end{tikzpicture}
%%%%%%%%%%%%%%%%%%%%%%%%%%%%%%%%%%%%%%%%
\caption{Four types of semi-arcs in a permutation diagram.}
\label{fig:fourtype2}
\end{center}
\end{figure}

\begin{theorem}
For every $n\ge 1$, the mapping $\Phi_{\YZL}:\SS_n\to \LH_n$ is a bijection.
\end{theorem}
\begin{proof}
We wish to prove this theorem by presenting the inverse of $\Phi_{\YZL}$. Given $W=(w,h,c) \in \LH_n$, it suffices to construct the permutation diagram of its corresponding permutation $\pi$. Our process begins with a line decorated with nodes labeled as $1,2,\ldots,n$ from left to right. Next, we take the following steps to assign for each node one outward semi-arc and one inward semi-arc. Set $k=\cs(W)$.
\begin{itemize}
  \item Assign a BL-arc for both $k$ and $n$, and a BR-arc for $1$.
  \item By the definition of the critical step, $w_{k}=\NE$. If $w_{k}=\N$, draw a BR-arc for $k+1$; if $w_k=\E$, draw an AL-arc for $k+1$.
  \item For $i\neq k$ from $1$ to $n-1$, draw the types (I)--(IV) of semi-arcs as given in Figure \ref{fig:fourtype} corresponding to $w_i=\N$, $\E$, $\dE$, and $\S$, respectively. This way, for the pair of nodes $(i,i+1)$, $i$ receives an outward semi-arc while $i+1$ receives an inward semi-arc.
  \item Use the weight $c_i$ and \eqref{def:vnest} to recover $\nest_i(\pi)$. Namely,
  \[\nest_i(\pi)=c_i-\chi(w_i=\SdE)+\chi(i \in \Sdeb)-\chi(i \in \Nea).\]
  \item Suppose that the corresponding labels of all the nodes with AR-arcs from left to right are $i_1,i_2,\ldots,i_s$, respectively. Link the AR-arc of $i_s$ to the $(\nest_{i_s}(\pi)+1)$-th AL-arc (counting from right to left). Link the AR-arc of $i_{s-1}$ to the $(\nest_{i_{s-1}}(\pi)+1)$-th AL-arc (counting form right to left within the remaining AL-arcs). Repeating like this until all the AR-arcs have been connected with the AL-arcs.
  \item Suppose that the corresponding labels of all the nodes with BL-arcs from left to right are $j_1,j_2,\ldots, j_t$, respectively. Link the BL-arc of $j_1$ to the $(\nest_{j_1}(\pi)+1)$-th BR-arc (counting from left to right). Link the BL-arc of $j_2$ to the $(\nest_{j_2}(\pi)+1)$-th BR-arc (counting from left to right within the remaining  BR-arcs). Repeating like this until all the BL-arcs have been connected with the BR-arcs.
\end{itemize}

In summary, the first three steps of our construction make sure each node is associated with one outward semi-arc and one inward semi-arc. The remaining steps tell us how these semi-arcs are connected so as to form a complete permutation diagram. Then, the desired permutation $\pi$ follows immediately from reading the permutation diagram. We should mention that strictly speaking, for the last two steps, some justifications are required to guarantee that the number of AR-arcs (resp.~BL-arcs) matches that of AL-arcs (resp.~BR-arcs). We verify this for one particular case and trust the reader with the other cases. Suppose $\cs(W)=k<n$ and $w_k=\N$, so $w_n=\S$ and $k+1$ receives a BR-arc according to step 2. Since $w$ represents a $2$-Motzkin path, the number of $\N$'s equals the number of $\S$'s in the remaining letters $w_i$, $i\neq k,i\neq n$. Therefore we have
\begin{align*}
&\phantom{=} \#\{i\in [n]: \text{$i$ is assigned an AR-arc}\}-\#\{i\in [n]: \text{$i$ is assigned an AL-arc}\}\\
&=\#\{i\in[n-1]\setminus\{k\}:\text{$(i,i+1)$ is of type (I) or (II)}\}\\
&\qquad\qquad -\#\{i\in[n-1]\setminus\{k\}:\text{$(i,i+1)$ is of type (II) or (IV)}\}\\
&=\#\{i\in[n-1]\setminus\{k\}:\text{$(i,i+1)$ is of type (I)}\}-\#\{i\in[n-1]\setminus\{k\}:\text{$(i,i+1)$ is of type (IV)}\}\\
&=\#\{i\in[n-1]\setminus\{k\}:w_i=\N\}-\#\{i\in[n-1]\setminus\{k\}:w_i=\S\}=0,
\end{align*}
as claimed.

Finally, it can be checked that indeed $\Phi_{\YZL}(\pi)=W$. We omit the details and the proof is now completed.
\end{proof}

Just like linear permutation statistics are transformed by $\Phi_{\FV}$ in section~\ref{subsec:FV}, and cyclic statistics are transformed by $\Phi_{\FZ}$ in section~\ref{subsec:FZ}, our new variant mapping $\Phi_{\YZL}$ deserves its own permutation statistics. This family of permutation statistics can be thought of as ``shifted cyclic statistics'' and we introduce them now.

For $\pi \in \SS_n$,  let
\begin{align*}
  \Nexcp(\pi)=\{i \in [n] : \pi(i) \leq i\}, &\,\,\,\, \Vnex(\pi)=\{ i \in [n-1]: i+1 \notin \Exc(\pi)\},\\[3pt]
  \Vnexpb(\pi)=\{i : i <\pone(\pi), i \in \Nexcp(\pi)\}, &\,\,\,\,
  \Vnexpa(\pi)=\{i : i >\pone(\pi), i \in \Nexcp(\pi)\}, \\[3pt]
  \Vnexb(\pi)=\{i : i <\pone(\pi), i \in \Vnex(\pi)\}, & \,\,\,\,\Vnexa(\pi)=\{i : i >\pone(\pi), i \in \Vnex(\pi)\}, \\[3pt]
  \Vexpb(\pi)=\{i : i <\pone(\pi), i \in \Excp(\pi)\}, &\,\,\,\, \Vexpa(\pi)=\{i : i >\pone(\pi), i \in \Excp(\pi)\}.
\end{align*}

Further, we define the multiset-valued statistics
{\em shifted excedance difference}, {\em shifted excedance position sum}, and {\em shifted nesting} as
\begin{align*}
\Vedif(\pi) &=\left(\bigcup_{i \in \Excp(\pi)}\{i+1,i+2, \ldots, \pi(i)-1\} \right) \bigcup \{\pone(\pi)+1,\ldots,n\},\\[3pt]
\Vbot(\pi) &=\bigcup_{i \in \Vnex(\pi)}\{i^{i}\},\,\,\,\,\,\,\,\,\text{and }
\Vnest(\pi)=\bigcup_{i=1}^n \,\,\{i^{\vnest_i(\pi)}\}.
\end{align*}

%$\Vedif(\pi)=\bigcup_{i \in \Excp(\pi)} {(i,\pi_i)} \cup (\pone(\pi),n]$
%and  $\Vnest(\pi)$ is the multiset containing $\vnest_i(\pi)$'s $i$.

\begin{example}
For $\sigma=671395482$, we have
$\Excp(\sigma)=\{1,2,5\}$,
$\Nexcp(\sigma)=\{3,4,6,7,8,9\}$,
$\Vnex(\sigma)=\{1,2,3,4,7\}$,
$\Vnexpb(\sigma)=\emptyset$, $\Vnexpa(\sigma)=\{4,6,7,8,9\}$,
$\Vexpb(\sigma)=\{1,2\}$, $\Vexpa(\sigma)=\{5\}$,
$\Vnexb(\sigma)=\{1,2\}$,
$\Vnexa(\sigma)=\{4,7\}$,
$\Vedif(\sigma)=\{2,3^2,4^3,5^3,6^3,7^2,8^2,9\}$,
$\Vbot(\sigma)=\{1,2^2,3^3,4^4,7^7\}$,
$\Vnest(\sigma)=\{4,5,6^2,7,8\}$.
\end{example}

\begin{proposition}\label{prop:YZL}
The bijection $\Phi_{\YZL}:\SS_n\rightarrow \LH_n$  links  permutation statistics with their counterparts over sr-Laguerre histories as follows. For any $\pi\in\SS_n$ and $W=(w,h,c)=\Phi_{\YZL}(\pi)$,
 we have
\begin{align}
\label{eq:YZLSET}
(\Vnexpb,\Vnexpa,\Vnexb,\Vnexa,\Vexpb,\Vexpa,\Vedif,\Vnexpb\sqcup\Vnexpa\sqcup\Vnest)\:\pi \\[3pt]
= (\Sdeb,\Sdea,\Ndeb,\Ndea,\Neb,\Nea, \Het,\Wt)\:W,  \nonumber
\end{align}
and
\begin{align}
&(\Vnest\sqcup \Vnexpb \setminus \Vexpa, \Vnest,
\Vedif\setminus\Vnexpb\setminus\Vnexpa\setminus\Vnest)\:\pi \\
&= ( \Wt\setminus\Nea\setminus\Sdea,
\Wt\setminus\Sdeb\setminus\Sdea,\Het\setminus\Wt)\:W. \nonumber
\end{align}
\end{proposition}

\begin{proof}
We just present the proof of $\Vedif(\pi)=\Het(W)$, while others can be verified through definitions and routine derivations.

Clearly, neither $\Vedif(\pi)$ nor $\Het(W)$ contains $1$. Suppose that the number of $i$'s contained in $\Vedif(\pi)$ is equal to that of $\Het(W)$. We wish to show that it also holds for $i+1$. There are three cases to consider:
\begin{itemize}
  \item If $h_{i+1}=h_{i}+1$, then $w_{i}=\N$ and
$|\Het(W)|_{i+1}=|\Het(W)|_{i}+1$. It suffices to prove that $|\Vedif(\pi)|_{i+1}=|\Vedif(\pi)|_{i}+1$. We consider two cases.
   \begin{itemize}
     \item  If $i \in \Scval(\pi)$ with $i\neq\pone(\pi)$, then $i < \pi(i)$ and $i+1 \leq \pi^{-1}(i+1)$. This implies that $i+1 \in (i,\pi(i))$, contributing one copy of $i+1$ to the multiset $\Vedif(\pi)$. Moreover, for every $j\in\Excp(\pi)$ with $1 \leq j < i$, the interval $(j,\pi(j))$ contains $i$ if and only if it contains $i+1$. Also, $(\pone(\pi),n]$ contains $i$  if and only if it contains $i+1$. So indeed $|\Vedif(\pi)|_{i+1}=|\Vedif(\pi)|_{i}+1$.
     \item  If $i=\pone(\pi)$ and $i+1 \leq \pi^{-1}(i+1)$, then $i \notin (\pone(\pi),n]$ and $i+1 \in (\pone(\pi),n]$. Moreover, for every $j\in\Excp(\pi)$ with $1 \leq j < i$, the interval $(j,\pi(j))$ contains $i$ if and only if it contains $i+1$, we therefore deduce that $|\Vedif(\pi)|_{i+1}=|\Vedif(\pi)|_{i}+1$ as well.
   \end{itemize}

  \item If $h_{i+1}=h_i$, then $w_i=\E$ or $w_i=\dE$ and $|\Het(W)|_{i+1}=|\Het(W)|_{i}$. We wish to show that $|\Vedif(\pi)|_{i+1}=|\Vedif(\pi)|_{i}$.
    \begin{itemize}
        \item If $w_i=\E$, we consider two cases.
        \begin{itemize}
            \item If $i \neq \pone(\pi)$, we have $i \in \Scda(\pi)$. I.e., $i < \pi(i)$ and $i+1 > \pi^{-1}(i+1)$. If $\pi(i)=i+1$, then for every $j\in\Excp(\pi)$ with $1 \leq j \le i$, we see that $(j,\pi(j))$ contains $i$ if and only if it contains $i+1$. Since $i \neq \pone(\pi)$, then $(\pone(\pi),n]$ contains $i$  if and only if it contains $i+1$.
            If $\pi(i) \neq i+1$, then $(\pi^{-1}(i+1),i+1)$ contains $i$ but not $i+1$, $(i,\pi(i))$ contains $i+1$ but not $i$, while other intervals contain $i$ if and only if they contain $i+1$. In both cases, we see that $|\Vedif(\pi)|_{i+1}=|\Vedif(\pi)|_{i}$, as desired.
            \item If $i=\pone(\pi)$, then we have $i+1>\pi^{-1}(i+1)$. Since $\pi(i) \neq i+1$, the interval $(\pi^{-1}(i+1),i+1)$ contains $i$ but not $i+1$. Also notice that in this case $(\pone(\pi),n]$ contains $i+1$ but not $i$. It follows from similar verifications with the remaining intervals that $|\Vedif(\pi)|_{i+1}=|\Vedif(\pi)|_{i}$.
        \end{itemize}

           \item If $w_i=\dE$, then we see that $i\in \Scdd(\pi)$ and $i \neq \pone(\pi)$. Since $i \geq \pi(i)$ and $i+1 \leq \pi^{-1}(i+1)$, we deduce that for every $j\in\Excp(\pi)$, the interval $(j,\pi(j))$ contains $i$ if and only if it contains $i+1$. Since $i \neq \pone(\pi)$, we see $(\pone(\pi),n]$ contains $i$ if and only if it contains $i+1$. So in this case $|\Vedif(\pi)|_{i+1}=|\Vedif(\pi)|_{i}$ holds as well.
    \end{itemize}

  \item If $h_{i+1}=h_i-1$, then $w_i=\S$ and
        $|\Het(W)|_{i+1}=|\Het(W)|_{i}-1$. We claim that $|\Vedif(\pi)|_{i+1}=|\Vedif(\pi)|_{i}-1$. Since $i \leq n-1$, we have $i\in \Scpk(\pi)$ and $i \neq \pone(\pi)$.
        Since $i\geq\pi(i)$ and $i+1>\pi^{-1}(i+1)$, we deduce that $(\pi^{-1}(i+1),i+1)$ contains $i$ but not $i+1$, while other intervals $(j,\pi(j))$ with $j\in\Excp(\pi)$ contain $i$ if and only if they contain $i+1$. So we see the claim holds true.
\end{itemize}

Combining all the cases above we finish the induction step, and thus come to the conclusion that $\Vedif(\pi)=\Het(W)$.
\end{proof}

%For $\pi \in S_n$, let $\vbot(\pi)=\sum_{i \in \Vnex(\pi)} i$. By combining Proposition \ref{prop:FV} and Proposition \ref{prop:YZL}, we obtain the following corollary.

\begin{corollary}
Let $\rho=\Phi_{\YZL}^{-1}\circ \xi \circ \Phi_{\YZL}$ and $\pi \in \SS_n$, then
\begin{align*}
  &(\Vnexpb,\Vnexpa,\Vexpb,\Vexpa,\Vnest \sqcup \Vnexpb \setminus \Vexpa ,\Vnest,\Vedif  \setminus \Vnexpb\setminus\Vnexpa \setminus \Vnest) \: \pi \\[4pt]
&=\kappa_{n+1} \circ (\Vexpa,\Vexpb,\Vnexpa,\Vnexpb,\Vnest,\Vnest \sqcup \Vnexpb \setminus \Vexpa,\\
&\qquad\qquad\qquad \Vedif  \setminus \Vnexpb\setminus\Vnexpa\setminus \Vnest) \:\rho(\pi),\\[4pt]
&[n-1] \setminus\,  \Vnex(\pi)=  \kappa_{n} \circ \Vnex(\rho(\pi)).
 % X(\pi)&=n+1-X \cup \Vnexpb \cup \Vexpb \setminus \Vnexpa \setminus \Vexpa \: (\rho (\pi)),\\[4pt]
%  X'(\pi)&=n+1-X'  \cup \Vexpb \setminus \Vnexpa \: (\eta (\pi)).
\end{align*}
%where
%\begin{align*}
 % X&=\Vedif \cup \Vnest, \\[4pt]
  %X'&=\Vedif \cup \Vedif \setminus \Nexcp \setminus \Vnest.
%\end{align*}
%are  set-valued Mahonian statistics.
\end{corollary}

%\begin{corollary}
%Statistics
%$\vedif+\vnest$, $\vbot+\vnest$ are Mahonian.
%\end{corollary}

Readers are recommended to check this corollary via the previous example $\sigma=671395482$ and its image $\rho(\sigma)=937628145$.

\subsection{Multiset-valued Mahonian statistics}\label{subsec:Mahonian}
Recall the classical permutation statistic \emph{inversion number}, defined for each $\pi\in\SS_n$ as $\inv(\pi)=\#\{1\le i<j\le n : \pi(i)>\pi(j)\}.$ A statistic that is equidistributed with $\inv$ is said to be \emph{Mahonian}.

As commented in \cite{CSZ}, new Mahonian statistics are ``constantly entering the scene''; see for example \cite{FZ90,SS,CSZ,BS}. Our present work is no exception. In this final subsection, we introduce seven apparently new Mahonian statistics. In fact, we derive first the multiset-valued statistics, and then take the cardinalities of the multisets to get these new Mahonian statistics. From this perspective, those multiset-valued counterparts are also said to be Mahonian. We note that all the multiset-valued Mahonian statistics considered here arise naturally as we consider the three mappings $\Phi_{\FV}$, $\Phi_{\FZ}$, $\Phi_{\YZL}$, and their compositions with our involution $\xi$, namely $\phi=\Phi_{\FV}^{-1}\circ \xi \circ \Phi_{\FV}$, $\eta=\Phi_{\FZ}^{-1}\circ \xi \circ \Phi_{\FZ}$, and $\rho=\Phi_{\YZL}^{-1}\circ \xi \circ \Phi_{\YZL}$.

Our starting point is the following four $\des$-based Mahonian statistics considered by Clarke, Steingr\'{i}msson and Zeng \cite{CSZ}, reformulated in our notations:
% In this subsection, we wish to introduce a series of Multiset-valued
% statistics, which we call Multiset-valued Mahonian statistics.
% As a consequence, by computing the cardinality of each multiset,
% we obtain the corresponding numerical Mahonian statistics.
% It turns out that some of the statistics are connected with
% known Mahonian statistics by complement, while others appear to be new.
\begin{align*}
   \mak(\pi) &=\dbot(\pi)+\ppwrs(\pi), \,\,\,\,\,\, \mad(\pi)=\ddif(\pi)+\ppwrs(\pi),   \\
 \makl(\pi)&=\dbot(\pi)+\pp31-2(\pi),\,\,\,\,\,\,
  \madl(\pi)  =\ddif(\pi)+ \pp31-2(\pi),
\end{align*}
where $\dbot(\pi)$ and $\ddif(\pi)$ are the cardinalities of the multisets $\Dbot(\pi)$ and $\Ddif(\pi)$, respectively. Note that $\mak$ was first defined by Foata and Zeilberger~\cite{FZ90}, while in \cite{CSZ} all of the above four statistics were extended to words.
Based on this, it seems natural to introduce the following four multiset-valued Mahonian statisitcs over permutations:
\begin{align}
\label{def:fourmulst1}
  \Mak &=\Dbot \sqcup \rightdescent,    \,\,\,\,\,\,\,
  \Mad =\Ddif \sqcup \rightdescent,  \\
  \Makl &=\Dbot \sqcup \leftdescent,    \,\,\,\,\,\,\,
  \Madl =\Ddif \sqcup \leftdescent.
  \label{def:fourmulst2}
\end{align}
% Clearly, we have $\#\Mak(\pi)=\mak(\pi)$, $\#\Mad(\pi)=\mad(\pi)$,
% $\#\Makl(\pi)=\makl(\pi)$ and $\#\Madl(\pi)=\madl(\pi)$.
% From this perspective,  we call those statistics   in
% (\ref{def:fourmulst1}) and (\ref{def:fourmulst2})
% the multiset-valued Mahonian statistics.

In view of (\ref{def:fourmulst1}), (\ref{def:fourmulst2}), and thanks to the equidistributions in Proposition \ref{prop:FV}, we are rewarded with the following four multiset-valued Mahonian statistics over sr-Laguerre
histories:
\begin{align}
\label{LH1}\overline{\Nde} &\sqcup (\Wt \setminus \Sdeb \setminus \Sdea),\\
\label{LH2}\Het &\sqcup (\Wt \setminus \Sdeb \setminus \Sdea),  \\
\label{LH3}\overline{\Nde} &\sqcup   (\Het \setminus \Wt),    \\
\label{LH4}\Het &\sqcup   (\Het \setminus \Wt),
\end{align}
where for a history $W\in\LH_n$, $\overline{\Nde}(W)$ is the multiset consisting of $i$ copies of $i$, for each $i \in \Nde(W)$. It should be noted that $\Dbot$ is mapped by $\Phi_{\FV}$
to $\overline{\Nde}$ as a result of $\Db$ being mapped by $\Phi_{\FV}$
to $\Nde$; see the footnote 2 on page 9.

On the other hand, by applying the involution $\xi$, we deduce that the following four companion multiset-valued statistics over sr-Laguerre histories are also Mahonian:
\begin{align}
\label{LH5} &\widetilde{\Nde }\sqcup \kappa_{n+1}\circ(\Wt \setminus \Nea \setminus \Sdea),\\
\label{LH6} &\kappa_{n+1}\circ((\Het \sqcup \Neb \setminus \Sdea)\sqcup (\Wt \setminus
\Nea \setminus \Sdea)),  \\
\label{LH7} &\widetilde{\Nde } \sqcup \kappa_{n+1}\circ(\Het \setminus \Wt),\\
\label{LH8} &\kappa_{n+1}\circ((\Het \sqcup \Neb \setminus \Sdea) \sqcup (\Het \setminus \Wt)).
\end{align}
Here for a history $W\in\LH_n$, the multiset $\widetilde{\Nde}(W)$ is composed of $n-i$ copies of $n-i$ for each $i \in [n-1]\setminus\Nde(W)$. Similarly, for a permutation $\pi\in\SS_n$, if we denote $\widetilde{\Db}(\pi)$ the multiset consisting of $n-i$ copies of $n-i$ for each $i\in[n-1]\setminus\Db(\pi)$, we get the following four multiset-valued Mahonian statistics over permutations, which correspond to \eqref{LH5}--\eqref{LH8} under the bijection $\Phi_{\FV}$.
\begin{align*}
\Mak' &:=\widetilde{\Db} \sqcup \kappa_{n+1}\circ\rightascent,\\
\Mad' &:=\kappa_{n+1}\circ((\Ddif\cup\Abb\setminus\Dta) \sqcup \rightascent),\\
\Makl' &:=\widetilde{\Db} \sqcup \kappa_{n+1}\circ\leftdescent,\\
\Madl' &:=\kappa_{n+1}\circ((\Ddif\sqcup\Abb\setminus\Dta) \sqcup \leftdescent).
\end{align*}

In the same vein, all multiset-valued permutation statistics equidistributed with the eight statistics \eqref{LH1}--\eqref{LH8} under the other two bijections $\Phi_{\FZ}$ and $\Phi_{\YZL}$ are automatically Mahonian. Thus, there are in total twenty four of them.

Associated with the mapping $\Phi_{\FZ}$ and relying on Proposition~\ref{prop:FZ}, we have\footnote{Note that the numerical statistics corresponding to $\Inv$ and $\Den$ are precisely the inversion number $\inv$ and the Denert's statistic $\den$~\cite{FZ90}, so we follow the convention to name the multiset-valued statistics $\Den$ and $\Inv$, rather than naming them $\FZ1$ and $\FZ2$.}
\begin{align*}
  \Den & :=\Ebot \sqcup \Ine, \\
  \Inv & :=\Edif \sqcup \Ine,\\
  \FZthree & :=\Ebot \sqcup
  (\Edif \setminus \Exc \setminus \Ine), \\
  \FZfour & :=\Edif \sqcup (\Edif \setminus \Exc \setminus \Ine),
\end{align*}
corresponding to \eqref{LH1}--\eqref{LH4}, while the following four are the counterparts of \eqref{LH5}--\eqref{LH8}:
\begin{align*}
\Den'& :=\widetilde{\Excp }\sqcup\kappa_{n+1}\circ(\Ine \, \sqcup  \Excb \setminus\Nexca),\\
\Inv'& :=\kappa_{n+1}\circ((\Edif \sqcup \Nexcb \setminus \Exca) \sqcup(\Ine \, \sqcup  \Excb \setminus\Nexca)),  \\
\FZthree'& :=\widetilde{\Excp  } \sqcup \kappa_{n+1}\circ(\Edif\setminus \Exc \setminus \Ine),\\
\FZfour'& :=\kappa_{n+1}\circ((\Edif \sqcup \Nexcb \setminus \Exca) \sqcup   (\Edif\setminus \Exc \setminus \Ine)),
\end{align*}
where again for a permutation $\pi\in\SS_n$, the multiset $\widetilde{\Excp}(\pi)$ is composed of $n-i$ copies of $n-i$ for each $i\in[n-1]\setminus\Excp(\pi)$.

The final eight multiset-valued Mahonian statistics are associated with the mapping $\Phi_{\YZL}$ and rely on Proposition~\ref{prop:YZL}. Here for a permutation $\pi\in\SS_n$, the multiset $\widetilde{\Vnex}(\pi)$ is composed of $n-i$ copies of $n-i$ for each $i\in[n-1]\setminus\Vnex(\pi)$.
\begin{align*}
\YZLone &:=\Vbot \sqcup \Vnest, \; \YZLtwo :=\Vedif \sqcup \Vnest, \\
\YZLthree &:=\Vbot \sqcup (\Vedif \setminus\Vnexpb\setminus\Vnexpa\setminus\Vnest),\\
\YZLfour &:=\Vedif \sqcup (\Vedif \setminus\Vnexpb\setminus\Vnexpa\setminus\Vnest), \\
  \YZLone' &:=\widetilde{\Vnex} \sqcup
\kappa_{n+1}\circ(\Vnest \sqcup\Vnexpb \setminus \Vexpa), \\
  \YZLtwo' &:=\kappa_{n+1}\circ((\Vedif \sqcup\Vexpb \setminus \Vnexpa) \sqcup (\Vnest \sqcup\Vnexpb \setminus \Vexpa)), \\
  \YZLthree' &:=\widetilde{\Vnex} \sqcup
   \kappa_{n+1}\circ(\Vedif \setminus\Vnexpb\setminus\Vnexpa\setminus\Vnest), \\
  \YZLfour' &:=\kappa_{n+1}\circ((\Vedif \sqcup\Vexpb \setminus \Vnexpa) \sqcup (\Vedif \setminus\Vnexpb\setminus\Vnexpa\setminus\Vnest)).
\end{align*}
% In particular, $\YZLtwo'(\pi)$ can be simplified as
% \begin{align}\label{YZL2prime}
% \YZLtwo'(\pi)=\kappa_{n+1}\circ\left(\Vnest(\pi)\sqcup(\sqcup_{i \in \Excp(\pi)}(i,\pi(i))) \sqcup[1,\pone(\pi))\right).
% \end{align}

Each multiset-valued Mahonian statistic induces a numerical Mahonian statistic, for which we denote by lowercase letters. Aside from those induced by \eqref{def:fourmulst1}, \eqref{def:fourmulst2}, $\Den$, and $\Inv$, we still have eighteen such Mahonian statistics. We catalog them (in some cases after simplification) in Table~\ref{Mahonian stats} and record these results in the following theorems and Appendix~\ref{sec:appendix}.

\begin{table}[htb]
{\small
\begin{tabular*}{5in}{ccc}
\toprule
Name  & Definition & Reference\\
\midrule\vspace{4pt}
$\maj$ \,\,\,\,\,\,\,\,& $(1\underline{32})+(2\underline{31})+(3\underline{21})+(\underline{21})$ \,\,\,\,\,\,\,\,& MacMahon \cite{M}\\\vspace{4pt}

$\inv$ \,\,\,\,\,\,\,\,& $(\underline{23}1)+(\underline{31}2)+(\underline{32}1)+(\underline{21})$ \,\,\,\,\,\,\,\,& MacMahon \cite{M}\\\vspace{4pt}

$\mak$\,\,\,\,\,\,\,\, & $(1\underline{32})+(2\underline{31})+(\underline{32}1)+(\underline{21})$ \,\,\,\,\,\,\,\,& Foata-Zeilberger\cite{FZ90}\\\vspace{4pt}

$\makl$ \,\,\,\,\,\,\,\,& $(1\underline{32})+(\underline{31}2)+(\underline{32}1)+(\underline{21})$ \,\,\,\,\,\,\,\,& Clarke-Steingr\'{i}msson-Zeng \cite{CSZ}\\\vspace{4pt}

$\mad$\,\,\,\,\,\,\,\, & $(2\underline{31})+(2\underline{31})+(\underline{31}2)+(\underline{21})$ \,\,\,\,\,\,\,\,& Clarke-Steingr\'{i}msson-Zeng \cite{CSZ} \\\vspace{4pt}

$\madl$\,\,\,\,\,\,\,\, & $(2\underline{31})+(\underline{31}2)+(\underline{31}2)+(\underline{21})$ \,\,\,\,\,\,\,\,& Clarke-Steingr\'{i}msson-Zeng \cite{CSZ} \\\vspace{4pt}

$\bast$ \,\,\,\,\,\,\,\,& $(\underline{13}2)+(\underline{21}3)+(\underline{32}1)+(\underline{21})$ \,\,\,\,\,\,\,\,& Babson-Steingr\'{i}msson \cite{BS} \\\vspace{4pt}

$\bast'$\,\,\,\,\,\,\,\, & $(\underline{13}2)+(\underline{31}2)+(\underline{32}1)+(\underline{21})$ \,\,\,\,\,\,\,\,& Babson-Steingr\'{i}msson \cite{BS} \\\vspace{4pt}

 $\bast''$ \,\,\,\,\,\,\,\,& $(1\underline{32})+(3\underline{12})+(3\underline{21})+(\underline{21})$ \,\,\,\,\,\,\,\,& Babson-Steingr\'{i}msson \cite{BS} \\\vspace{4pt}

 $\foze$ \,\,\,\,\,\,\,\,& $(\underline{21}3)+(3\underline{21})+(\underline{13}2)+(\underline{21})$ \,\,\,\,\,\,\,\,& Foata-Zeilberger \cite{FZ01}\\\vspace{4pt}

 $\foze'$ \,\,\,\,\,\,\,\,& $(1\underline{32})+(2\underline{31})+(2\underline{31})+(\underline{21})$ \,\,\,\,\,\,\,\,& Foata-Zeilberger \cite{FZ01}\\\vspace{4pt}

  $\foze''$ \,\,\,\,\,\,\,\,& $(\underline{23}1)+(\underline{31}2)+(\underline{31}2)+(\underline{21})$ \,\,\,\,\,\,\,\,& Foata-Zeilberger \cite{FZ01} \\\vspace{4pt}

  $\sist$\,\,\,\,\,\,\,\, & $(\underline{13}2)+(\underline{13}2)+(2\underline{13})+(\underline{21})$ \,\,\,\,\,\,\,\,& Simion-Stanton \cite{SS}\\\vspace{4pt}

  $\sist'$ \,\,\,\,\,\,\,\,& $(\underline{13}2)+(\underline{13}2)+(2\underline{31})+(\underline{21})$ \,\,\,\,\,\,\,\,& Simion-Stanton \cite{SS}\\\vspace{4pt}

  $\sist''$ \,\,\,\,\,\,\,\,& $(\underline{13}2)+(2\underline{31})+(2\underline{31})+(\underline{21})$ \,\,\,\,\,\,\,\,& Simion-Stanton \cite{SS}\\\vspace{4pt}

  $\den$ \,\,\,\,\,\,\,\,& $\ebot+\ine$ \,\,\,\,\,\,\,\,& Denert \cite{Den}\\\vspace{4pt}

  $\sort$ \,\,\,\,\,\,\,\,&
  \tabincell{c}{\small{$\pi=(i_1,j_1)(i_2,j_2)\cdots (i_k,j_k),$}\\
  \small{$\sort(\pi)=\sum_{r=1}^{k}(j_r-i_r)$}}
  \,\,\,\,\,\,\,\,& Petersen \cite{P}\\

\bottomrule\\
\end{tabular*}
}
\caption{Seventeen known Mahonian statistics defined for $\pi=\pi_1\cdots\pi_n\in\SS_n$}\label{known Mahonian stats}
\end{table}

\begin{remark}
For a certain statistic in Table~\ref{Mahonian stats} that is commented with, say $n=3$, we mean that as our computer-asisted verification gets to permutations of $n=3$ letters, this statistic is different from all seventeen Mahonian statistics listed in Table~\ref{known Mahonian stats} (we have largely followed the nomenclature of Amini~\cite[Tab.~1]{Am} and corrected two typos thereof in the definitions of $\mak$ and $\makl$), as well as their images under the action of the dihedral group $D_4$ that is generated by the three trivial bijections: reverse ($\r:\pi\mapsto\pi^{\r}$), complement ($\cc:\pi\mapsto\pi^{\cc}$), and inverse ($\i:\pi\mapsto\pi^{\i}$)~\cite[Defn.~1.0.12]{Kit}. We are inclined to believe these seven Mahonian statistics are new, but whether or not they have already appeared in the literature in some disguise remains to be seen.
\end{remark}

\begin{theorem}\label{thm:Mahon stats}
The permutation statistics listed in Table~\ref{Mahonian stats} are all Mahonian. In particular, for each permutation $\pi$, we have
\begin{align}
\label{mad=sist}\mad'(\pi) &= \sist''(\pi^{\cc}),\\
\label{madl=sist}\madl'(\pi) &= \sist'(\pi^{\cc}),\\
\label{makl'=makl^c}\makl'(\pi) &= \makl(\pi^{\cc}).
\end{align}
\end{theorem}

It suffices to show \eqref{mad=sist}--\eqref{makl'=makl^c}. To this end, we need the following two lemmas.
\begin{lemma}
For each permutation $\pi=\pi_1\cdots\pi_n$, we have
\begin{align}
\label{id:213-231}
(\pprs+\mathrm{\underline{12}})\:\pi = (\ppwrs)\:\pi+\pi_n-1.
\end{align}
\end{lemma}
\begin{proof}
For each $1\le i<n$, there is a unique way to decompose the suffix $\pi_{i+1}\cdots\pi_n$ into factors $A_1,A_2,\ldots$ and $B_1,B_2,\ldots$, such that the $A$-labeled (resp.~$B$-labeled) factors are consisted of letters larger than (resp.~smaller than) $\pi_i$. Then, one of the following four cases must occur, for a certain $k\ge 1$.
\begin{enumerate}[(I)]
    \item $\pi_{i+1}\cdots\pi_n=B_1A_1\cdots B_kA_k$;
    \item $\pi_{i+1}\cdots\pi_n=B_1A_1\cdots B_{k-1}A_{k-1}B_k$;
    \item $\pi_{i+1}\cdots\pi_n=A_1B_1\cdots A_kB_k$;
    \item $\pi_{i+1}\cdots\pi_n=A_1B_1\cdots A_{k-1}B_{k-1}A_k$.
\end{enumerate}
Now we compute the contributions to both sides of \eqref{id:213-231}. Together with two adjacent letters from neighboring $A$ and $B$ factors, $\pi_i$ contributes $1$ to either $\pprs$ or $\ppwrs$. It contributes $\chi(\pi_i<\pi_{i+1})$ to $\mathrm{\underline{12}}$ and $\chi(\pi_i<\pi_n)$ to $\pi_n-1$, respectively. So for instance, suppose we are in case (I), then the contributions to the left hand side of \eqref{id:213-231} is $k+0=k$, while to the right hand side of \eqref{id:213-231} is $k-1+1=k$ as well. In the same vein, the other three cases can be checked. Finally, $\pi_n$ contributes nothing to either side and we are done.
\end{proof}

\begin{lemma}
For each permutation $\pi=\pi_1\cdots\pi_n$, we have
\begin{align}
\label{id:all complement}
(3\underline{12}+\underline{12}3+\pprs+\underline{13}2+\underline{12})\pi+n\cdot\des\pi &-(1\underline{32}+\underline{32}1+\ppwrs+\pp31-2+2\cdot\underline{21})\pi\\
&=\frac{n(n-3)}{2}+\pi_n.\nonumber
\end{align}
\end{lemma}
\begin{proof}
The right hand side reads
$$n(n-3)/2+\pi_n=n(n-1)/2-(n-\pi_n)=\sum_{1\le i<n}(n-\pi_i).$$
For the left hand side, we observe that for a fixed $\pi_i$, $1\le i<n$, if $\pi_i<\pi_{i+1}$, then the contribution of $n-\pi_i$ is distributed among the summands of
$$(3\underline{12}+\underline{12}3+\pprs+\underline{13}2+\underline{12})\pi,$$
where $\pi_i$ always plays the role of $1$; if $\pi_i>\pi_{i+1}$, then the contribution of $n-\pi_i$ is accounted by the remaining terms
$$n\cdot\des\pi -(1\underline{32}+\underline{32}1+\ppwrs+\pp31-2+2\cdot\underline{21})\pi.$$
Thus \eqref{id:all complement} is established via double counting.
\end{proof}

\begin{proof}[Proof of Theorem~\ref{thm:Mahon stats}]
For \eqref{mad=sist}, we express both $\mad'$ and $\sist''$ using vincular patterns and apply \eqref{id:213-231} to deduce that (also note that $(\underline{21})\pi+(\underline{12})\pi=n-1$.)
\begin{align*}
\mad'(\pi) &= (2\cdot\ppwrs+\pp31-2+\mathrm{\underline{21}})\pi+2\pi_n-n-1\\
&= 2((\pprs+\mathrm{\underline{12}})\pi)+(\pp31-2+\mathrm{\underline{21}})\pi-n+1\\
&= (2\cdot\pprs+\pp31-2+\mathrm{\underline{12}})\pi\\
&= \sist''(\pi^{\cc}).
\end{align*}
In view of $\madl(\pi)=\mad((\pi^{\r})^{\cc})$, the derivation of \eqref{madl=sist} goes similarly and also relies on \eqref{id:213-231}. Next, for \eqref{makl'=makl^c}, we express both $\makl$ and $\makl'$ in terms of vincular patterns, and utilize both \eqref{id:all complement} and \eqref{id:213-231} to deduce that
\begin{align*}
\makl(\pi^{\cc}) &= (3\underline{12}+\underline{12}3+\underline{13}2+\underline{12})\pi\\
&= \frac{n(n-3)}{2}+\pi_n-n\cdot\des\pi-(\pprs)\pi+(1\underline{32}+\underline{32}1+\ppwrs+\pp31-2+2\cdot\underline{21})\pi\\
&= \frac{n(n-1)}{2}-n\cdot\des\pi+(1\underline{32}+\underline{32}1+\pp31-2+\underline{21})\pi\\
&= \makl'(\pi).
\end{align*}
\end{proof}

%%%%%%%%% Table 3 %%%%%%%%%%%%%%%%%%%%%%%%%
%@{\extracolsep{\fill}}
\begin{table}
{\small
\begin{tabular*}{5in}{ccc}
\toprule
statistic & definition & comment\\
\midrule
$\mak'(\pi)$ & $\mak(\pi)+(1-n)\des(\pi)+\pi_n+n(n-3)/2$ & $n=4$\\
&&\\
$\mad'(\pi)$ & $\mad(\pi)+2\pi_n-n-1$ & \eqref{mad=sist}\\
&&\\
$\makl'(\pi)$ & $\makl(\pi)-n\cdot\des(\pi)+n(n-1)/2$ & \eqref{makl'=makl^c}\\
&&\\
$\madl'(\pi)$ & $\madl(\pi)-\des(\pi)+\pi_n-1$ & \eqref{madl=sist}\\
&&\\
$\fzthree(\pi)$ & $\ebot(\pi)+\edif(\pi)-\exc(\pi)-\ine(\pi)$ & $n=3$\\
&&\\
$\fzfour(\pi)$ & $2\edif(\pi)-\exc(\pi)-\ine(\pi)$ & $n=4$\\
&&\\
$\inv'(\pi)$ & $\inv(\pi)+2\pi_n-1-n$ & $n=3$\\
&&\\
$\den'(\pi)$ & $\den(\pi)+(1-n)\exc(\pi)+\pi_n+n(n-3)/2$ & $n=3$\\
&&\\
$\fzthree'(\pi)$ & $\fzthree(\pi)-n\cdot\exc(\pi)+n(n-1)/2$ & $n=3$\\
&&\\
$\fzfour'(\pi)$ & $\fzfour(\pi) -\exc(\pi) +\pi_n-1$ & $n=3$\\
&&\\
$\yzlone(\pi)$ & $\vbot(\pi) +\vnest(\pi)$ & \eqref{yzl1=den}\\
&&\\
$\yzltwo(\pi)$ & $\vedif(\pi) +\vnest(\pi)$ & \eqref{yzl1=den}\\
&&\\
$\yzlthree(\pi)$ & $\vbot(\pi)+\vedif(\pi)-(n-2-\exc(\pi))-\vnest(\pi)$ & \eqref{yzl1=den}\\
&&\\
$\yzlfour(\pi)$ & $2\vedif(\pi)-(n-2-\exc(\pi))-\vnest(\pi)$ & \eqref{yzl1=den}\\
&&\\
$\yzlone'(\pi)$ & $\yzlone(\pi)+(n-1)\exc(\pi)+\pone(\pi)+(-n^2+n-2)/2$ & \eqref{yzl1=den}\\
&&\\
$\yzltwo'(\pi)$ & $\yzltwo(\pi)+2\pone(\pi)-n-1$ & \eqref{yzl1=den}\\
&&\\
$\yzlthree'(\pi)$ & $\yzlthree(\pi)+n\cdot\exc(\pi)-n(n-1)/2$ & \eqref{yzl1=den}\\
&&\\
$\yzlfour'(\pi)$ & $\yzlfour(\pi)+\exc(\pi)+\pone(\pi)-n$ & \eqref{yzl1=den}\\
\bottomrule\\
\end{tabular*}
}
\caption{Eighteen Mahonian statistics defined for $\pi=\pi_1\cdots\pi_n\in\SS_n$}\label{Mahonian stats}
\end{table}

\begin{remark}
As an indication of how one might prove individually a certain statistic listed in Table~\ref{Mahonian stats} is Mahonian, we sketch here a proof that $\inv'$ is Mahonian. Since Mahonian statistics must have the distribution over $\SS_n$ given by $(1+q)(1+q+q^2)\cdots(1+q+\cdots+q^{n-1})$, it suffices to explain the extra factor $(1+q+\cdots+q^{n-1})$ when one generates a permutation $\sigma\in\SS_n$ by inserting the letter $n$ into a permutation $\pi\in\SS_{n-1}$. Since the definition of $\inv'$ involves the last entry of the permutation, we consider two cases separately. Firstly, if $\sigma$ ends with $n$, then $\inv'(\sigma)=\inv(\sigma)+2n-1-n=\inv(\pi)+n-1$. And the subset containing all such $\sigma$'s is clearly in bijection with $\SS_{n-1}$, thus we have
\begin{align*}
\sum_{\sigma\in\SS_n,~\sigma(n)=n}q^{\inv'(\sigma)}&=\sum_{\pi\in\SS_{n-1}}q^{\inv(\pi)+n-1}=q^{n-1}(1+q)(1+q+q^2)\cdots(1+q+\cdots+q^{n-2}).
\end{align*}
Otherwise $\sigma(n)\neq n$, there are $n-1$ possible slots to insert $n$ in $\pi$ that will result in such a $\sigma$. We label them as follows:
$${}_{n-2}\pi(1)_{n-3}\pi(2)_{n-4}\cdots_{1}\pi(n-2)_0\pi(n-1).$$
Now it is easy to check that inserting $n$ at slot $j$ leads to $\inv'(\sigma)=\inv'(\pi)+j$, for every $0\le j\le n-2$. Combining two cases finishes the proof.
\end{remark}

As for the eight Mahonian statistics associated with the mapping $\Phi_{\YZL}$, our computation suggests the following relation between them and those associated with the mapping $\Phi_{\FZ}$. For each permutation $\pi\in\SS_n$, we have
\begin{align}
\label{yzl1=den}
(\yzlone,\yzltwo,\yzlthree,\yzlfour,\yzlone',\yzltwo',\yzlthree',\yzlfour')\:\pi &= (\den',\inv',\fzthree',\fzfour',\den,\inv,\fzthree,\fzfour)\:\pi^{\rci},
\end{align}
where $\pi^{\rci}:=\r\circ\cc\circ\i(\pi)$.

In seeking a proof of \eqref{yzl1=den}, we are led to the next theorem, which reveals further (and somewhat surprising) relations among the mappings considered in this paper. We introduce two mappings, defined for a permutation $\pi=\pi(1)\pi(2)\cdots\pi(n)$. The first mapping is ``almost'' the composition $\r\circ\cc$:
\begin{align*}
\theta(\pi):=\theta_1\cdots\theta_n,\text{ where } \theta_n=n+1-\pi(n),~\theta_i=n+1-\pi(n-i),\text{ for } 1\le i<n,
\end{align*}
while the second mapping
$$\varkappa(\pi):= \pi^{-1}(2)\cdots\pi^{-1}(n)\pi^{-1}(1)$$
is usually referred to as the {\it Kreweras complement} \cite{Kre}; see also \cite[Chap.~4.2]{Arm} and \cite{SU}.

\begin{theorem}\label{thm:yzl family}
For each permutation $\pi$, we have
\begin{align}\label{id:xi and theta}
&\Phi^{-1}_{\FZ}\circ \xi\circ \Phi_{\FZ}(\pi)=\theta(\pi),\\
&\Phi_{\YZL}(\pi)=\Phi_{\FZ}\circ\theta(\pi^{\rci})=\Phi_{\FZ}\circ\varkappa(\pi).
\label{id:Kreweras}
\end{align}
% Alternatively, viewing that, identity \eqref{conjid1} is equivalent to the following identity:
% \begin{align}
% \Phi^{-1}_{\FZ}\circ\xi\circ\Phi_{\YZL}(\pi)=\pi^{\rci}.
% \end{align}
% where
\end{theorem}
Combining \eqref{id:xi and theta} and \eqref{id:Kreweras}, we see that
$$\Phi_{\YZL}(\pi)=\xi\circ\Phi_{\FZ}(\pi^{\rci}),$$
which immediately yields \eqref{yzl1=den}. The proof of Theorem \ref{thm:yzl family} is a bit onerous so we decide to put it in the appendix.

Comparing to the classic mappings $\Phi_{\FV}$ and $\Phi_{\FZ}$, the mapping $\Psi_{\YZL}$ \cite{YZL}, as well as its variant $\Phi_{\YZL}$ introduced in this paper are quite recent. The work of Han-Mao-Zeng \cite{HMZ} indicates that $\Psi_{\YZL}$ actually fits well with $\Psi_{\FV}$ by placing $\Psi_{\YZL}$ in the first triangle in Fig.~\ref{two triangle}, where the mapping $\Psi_{\SZ}$ was first constructed by Shin-Zeng \cite{SZ12}. Viewing these two factorizations in Fig.~\ref{two triangle}, one wonders if a similar factorization could be carried out for $\Phi_{\FZ}$ and $\Phi_{\YZL}$.\footnote{Here one should consider our variant $\Phi_{\YZL}$ rather than the original $\Psi_{\YZL}$, since the image should be $\LH_n$, not $\mathtt{L}_n$.} The second equality in \eqref{id:Kreweras} makes such factorization explicit, and we are rewarded with a third triangle in Fig.~\ref{3rd triangle}.

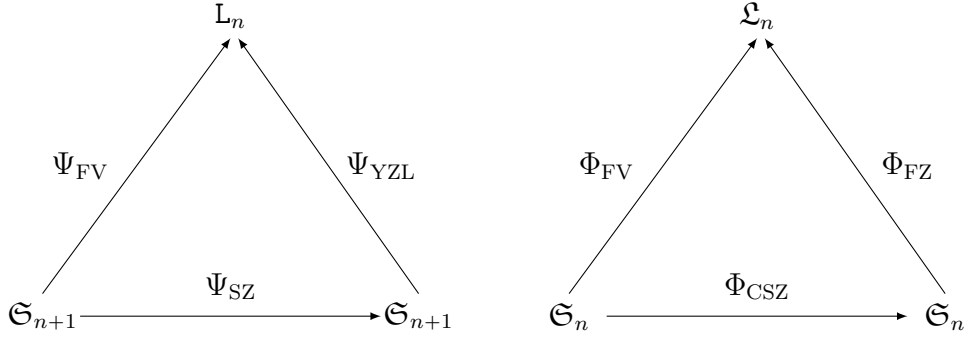
\begin{figure}[htb]
\begin{tikzpicture}[scale=1]
%% nodes
\draw(0,0) node{$\SS_{n+1}$};
\draw(5,0) node{$\SS_{n+1}$};
\draw(2.5,4) node{$\mathtt{L}_n$};

\draw(7,0) node{$\SS_{n}$};
\draw(12,0) node{$\SS_{n}$};
\draw(9.5,4) node{$\LH_n$};

\draw(.5,2) node{$\Psi_{\FV}$};
\draw(4.5,2) node{$\Psi_{\YZL}$};
\draw(2.5,0.4) node{$\Psi_{\SZ}$};

\draw(7.5,2) node{$\Phi_{\FV}$};
\draw(11.5,2) node{$\Phi_{\FZ}$};
\draw(9.5,0.4) node{$\Phi_{\CSZ}$};

%% edges
\draw[-latex] (0,0.3)->(2.5,3.7);
\draw[-latex] (0.5,0)->(4.5,0);
\draw[-latex] (5,0.3)->(2.6,3.7);
\draw[-latex] (7,0.3)->(9.5,3.7);
\draw[-latex] (7.5,0)->(11.5,0);
\draw[-latex] (12,0.3)->(9.6,3.7);
\end{tikzpicture}
\caption{Two factorizations: $\Psi_{\FV}=\Psi_{\YZL}\circ\Psi_{\SZ}$ and $\Phi_{\FV}=\Phi_{\FZ}\circ\Phi_{\CSZ}$. \label{two triangle}}
\end{figure}

\begin{figure}[htb]
\begin{tikzpicture}[scale=1]
%% nodes
\draw(0,0) node{$\SS_{n}$};
\draw(5,0) node{$\SS_{n}$};
\draw(2.5,4) node{$\LH_n$};

\draw(.5,2) node{$\Phi_{\YZL}$};
\draw(4.5,2) node{$\Phi_{\FZ}$};
\draw(2.5,0.4) node{$\varkappa$};

%% edges
\draw[-latex] (0,0.3)->(2.5,3.7);
\draw[-latex] (0.5,0)->(4.5,0);
\draw[-latex] (5,0.3)->(2.6,3.7);
\end{tikzpicture}
\caption{A third factorization: $\Phi_{\YZL}=\Phi_{\FZ}\circ\varkappa$. \label{3rd triangle}}
\end{figure}
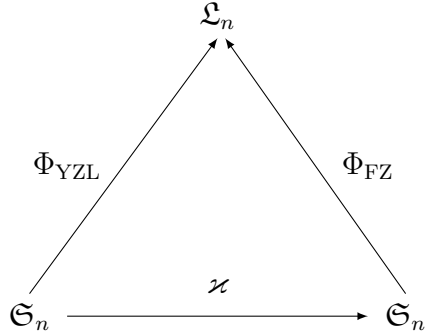

%%%%%%%%%%%%%%%%%%%%%%%%%%%%%%%%%%
\appendix
\section{The proof of Theorem~\ref{thm:yzl family}}\label{sec:appendix}
\begin{proof}[Proof of \eqref{id:xi and theta}]

We aim to give the proof by showing  that $\Phi_{\FZ} \circ \theta \circ \Phi_{\FZ}^{-1}(W)=\xi(W)$ for any $W\in \LH_n$.
Assume that $\Phi_{\FZ}^{-1}(W)=\pi$, $\sigma=\theta(\pi)$ and $V=\Phi_{\FZ}(\sigma)$ with $W=(w,h,c)$ and $V=(v,g,b)$.
It suffices to verify that $\Phi_{\FZ} \circ \theta \circ \Phi_{\FZ}^{-1}(W)=V$ satisfies conditions $(1) - (4)$ in Theorem \ref{thm:inv-xsi}.

Condition $(1)$ follows directly from Proposition
\ref{prop:FZ} and the definition of $\theta$.  To confirm condition $(2)$,  we assume that $j \in [n]$ and $j \neq \cs(W)$.
If $w_j=\NE$, then $\pi^{-1}(j) \geq j$. By the definition of $\theta$,
we have
\[ \sigma^{-1}(n+1-j) = n-\pi^{-1}(j) \leq n-j <n+1-j.\]
This indicates that $n+1-j \in \Cda(\sigma)$ or $\Cpk(\sigma)$,
namely, $v_{n+1-j}=\SdE$. Notice that $\theta$ is an involution, and hence $\Phi_{\FZ} \circ \theta \circ \Phi_{\FZ}^{-1}$ is an involution.
Based on this, 
we deduce that if $w_j=\SdE$, then $v_{n+1-j}=\NE$ and condition $(2)$
follows.
For condition $(3)$, firstly, we wish to explore in two ways what corresponds to $h_i$ under the map $\Phi_{\FZ}$:
\begin{align*}
  h_i & =\#\{j<i: w_j=\N\}- \#\{j<i: w_j=\S\}\\[3pt]
    & =\#\{j<i: j \in \Cval(\pi)\}- \#\{j<i: j \in \Cpk(\pi)\}\\[3pt]
     & =\#\{j<i: j \in \Cval(\pi)\}+\#\{j<i: j \in \Cdd(\pi)\}-\#\{j<i: j \in \Cdd(\pi)\}\\[3pt]
     &\,\,\,\,\,\,\,\,\,\,\,\,\,\,\,- \#\{j<i: j \in \Cpk(\pi)\}\\
     & =\#\{j<i: j \in \Nexc(\pi)\}-\#\{j<i: j \in \Nexcp(\pi)\},
\end{align*}
or
\begin{align*}
     h_i  & =\#\{j<i: j \in \Cval(\pi)\}+\#\{j<i: j \in \Cda(\pi)\}-\#\{j<i: j \in \Cda(\pi)\}\\[3pt]
     &\,\,\,\,\,\,\,\,\,\,\,\,\,\,\,- \#\{j<i: j \in \Cpk(\pi)\}\\
     & =\#\{j<i: j \in \Excp(\pi)\}-\#\{j<i: j \in \Exc(\pi)\}.
\end{align*}
Secondly, we investigate what happens to the above interpretation under the map $\theta$.
\begin{align*}
     g_i  & =\#\{j<i: j \in \Excp(\sigma)\}-\#\{j<i: j \in \Exc(\sigma)\}\\[3pt]
     & =\#\{j<i: j \in \Excp(\sigma)\}-\#\{j\leq i: j \in \Exc(\sigma)\}
     +\chi(i \in \Exc(\sigma))\\[3pt]
     &=\#\{n-j \geq n-i+1: n-j \in \Nexcp(\pi)\setminus \{n\}\}\\[3pt]
   &\,\,\,\,\,\,\,\,\,\,\,\,\,\,\,-\#\{n-j+1 \geq n-i+1: n-j+1 \in \Nexc(\pi)\setminus \{\pi_n\}\}
    +\chi(i \in \Exc(\sigma))\\[3pt]
   &=   \nep(\pi) -1-\#\{n-j < n-i+1: n-j \in \Nexcp(\pi)\setminus \{n\}\}-\nexc(\pi)+1 \\[3pt]
   &\,\,\,\,\,\,\,\,\,\,\,\,\,\,\,+\#\{n-j+1 < n-i+1: n-j+1 \in \Nexc(\pi)\setminus \{\pi_n\}\} + \chi(i \in \Exc(\sigma))   \\[3pt]
   &=  \#\{j \le n-i: j \in \Nexc(\pi)\setminus \{\pi_n\}\} -\#\{j \le n-i: j \in \Nexcp(\pi)\setminus \{n\}\} + \chi(i \in \Exc(\sigma)).
\end{align*}

Notice that $h_{n+1-i}=\#\{j\le n-i: j \in \Nexc(\pi)\}-\#\{j\le n-i: j \in \Nexcp(\pi)\}$ and $i \in \Exc(\sigma)$ if and only if $v_i=\SdE$, we consider four cases:
\begin{itemize}
  \item [I.] If $v_i=\SdE$ and $i<n+1-\pi_n$, then $\chi(i \in \Exc(\sigma))=1$ and $n-i\ge \pi_n$. Since  $\pi_n \in \Nexc(\pi)$, we have $g_i=h_{n+1-i}-1+1=h_{n+1-i}$.
  \item [II.]If $v_i=\SdE$ and $i>n+1-\pi_n$, then $\chi(i \in \Exc(\sigma))=1$ and $n+1-i<\pi_n$. It follows that 
  $g_i=h_{n+1-i}+1$.
  \item [III.]If $v_i=\NE$ and $i<n+1-\pi_n$, then $\chi(i \in \Exc(\sigma))=0$ and $n-i\ge \pi_n$. It follows that
  $g_i=h_{n+1-i}-1$.
  \item [IV.]If $v_i=\NE$ and $i \geq n+1-\pi_n$, then $\chi(i \in \Exc(\sigma))=0$ and $n+1-i \leq \pi_n$. Thus, we have 
  $g_i=h_{n+1-i}$.
\end{itemize}
Combining the above cases, condition $(3)$ is verified.
As for condition $(4)$, we assume that $s(\pi)=s_1(\pi)s_2(\pi)\cdots s_n(\pi)$ is the side number of $\pi$. By the definition of $\theta$, it is easy to check that

\begin{equation}\label{eq:srelation}
s_{\sigma^{-1}(i)}(\sigma)= \left\{
  \begin{array}{ll}
s_{\pi^{-1}(n+1-i)}(\pi), & \mbox{ $ i>n+1-\pi_n,\, v_i=\SdE$,}\\[3pt]
s_{\pi^{-1}(n+1-i)}(\pi)-1, & \mbox{ $ i<n+1-\pi_n,\, v_i=\SdE$,}\\[3pt]
s_{\pi^{-1}(n+1-i)}(\pi)+1, & \mbox{ $ i>n+1-\pi_n,\, v_i=\NE$,}\\[3pt]
s_{\pi^{-1}(n+1-i)}(\pi),  & \mbox{ $ i\leq n+1-\pi_n,\, v_i=\NE$.}
 \end{array} \right.
\end{equation}

In the following, we consider five cases.
\begin{itemize}
  \item [I.] If $v_i=\SdE$ and $i<n+1-\pi_n$, then $w_{n+1-i}=\NE$ and by condition $(3)$ we have $g_i=h_{n+1-i}$. Notice that $b_i=s_{\sigma^{-1}(i)}(\sigma)+1$ and  $c_{n+1-i}=s_{\pi^{-1}(n+1-i)}(\pi)$. In view of (\ref{eq:srelation}), we deduce that $b_i=c_{n+1-i}$.
      It follows that $b_i=g_i-h_{n+1-i}+c_{n+1-i}$.
      
  \item [II.]If $v_i=\SdE$ and $i>n+1-\pi_n$, then $w_{n+1-i}=\NE$ and $g_i=h_{n+1-i}+1$. Since $b_i=s_{\sigma^{-1}(i)}(\sigma)+1$ and  $c_{n+1-i}=s_{\pi^{-1}(n+1-i)}(\pi)$, we deduce that $b_i=c_{n+1-i}+1$ by (\ref{eq:srelation}). Thus,  $b_i=g_i-h_{n+1-i}+c_{n+1-i}$ follows.
      
  \item [III.]If $v_i=\NE$ and $i<n+1-\pi_n$, then $w_{n+1-i}=\SdE$ and  $g_i=h_{n+1-i}-1$. Since $b_i=s_{\sigma^{-1}(i)}(\sigma)$, $c_{n+1-i}=s_{\pi^{-1}(n+1-i)}(\pi)+1$, and in view of (\ref{eq:srelation}), we have $b_i=c_{n+1-i}-1$. Thus, $b_i=g_i-h_{n+1-i}+c_{n+1-i}$.
      
  \item [IV.]If $v_i=\NE$ and $i > n+1-\pi_n$, then $w_{n+1-i}=\SdE$  and  $g_i=h_{n+1-i}$. Since $b_i=s_{\sigma^{-1}(i)}(\sigma)$, $c_{n+1-i}=s_{\pi^{-1}(n+1-i)}(\pi)+1$, and in view of (\ref{eq:srelation}), we deduce that $b_i=c_{n+1-i}$, and hence $b_i=g_i-h_{n+1-i}+c_{n+1-i}$.
      
  \item [V.]If $v_i=\NE$ and $i = n+1-\pi_n$, then $w_{n+1-i}=\NE$  and  $g_i=h_{n+1-i}$. Since $b_i=s_{\sigma^{-1}(i)}(\sigma)$, $c_{n+1-i}=s_{\pi^{-1}(n+1-i)}(\pi)$, and in view of (\ref{eq:srelation}), we deduce that $b_i=c_{n+1-i}$. Thus, we also have $b_i=g_i-h_{n+1-i}+c_{n+1-i}$.
\end{itemize}
Combining all the cases, we confirm condition $(4)$ and the proof is now completed.
\end{proof}

\begin{proof}[Proof of \eqref{id:Kreweras}]
The fact that $\theta(\pi^{\rci})=\varkappa(\pi)$ is readily verified using definitions. It suffices then to show that $\Phi_{\YZL}(\pi)=\Phi_{\FZ}\circ\varkappa(\pi)$. Suppose $W=(w,h,c)=\Phi_{\YZL}(\pi)$ and $V=(v,g,b)=\Phi_{\FZ}\circ\varkappa(\pi)$. We aim to show that $w=v$ and $c=b$. For a fixed $w_i$, $1\le i\le n$, we consider the following four cases according to the type of this $i$-th step. The following two-line notation of $\varkappa(\pi)$ will be convenient as we determine $v_i$:
\begin{align*}
\varkappa(\pi)=\begin{pmatrix}
1 & 2 & \cdots & n-1 & n\\
\pi^{-1}(2) & \pi^{-1}(3) & \cdots & \pi^{-1}(n) & \pi^{-1}(1)
\end{pmatrix}.
\end{align*}
\begin{enumerate}[(I)]
    \item $w_i=\N$, then $i<n$. If $i\neq \pone(\pi)$, then $i\in\Scval(\pi)$, i.e., $i<\pi(i)$ and $i+1\le \pi^{-1}(i+1)$, which are equivalent to
    $$i\le\pi(i)-1,\text{ and } i<\pi^{-1}(i+1).$$
    This implies that $i\in\Cval(\varkappa(\pi))$, thus $v_i=\N$ as well. Otherwise $i=\pone(\pi)$, in other words $i=\pi^{-1}(1)$, and $i+1\le\pi^{-1}(i+1)$ (since $w_i=\N$), or equivalently $i<\pi^{-1}(i+1)$. So we still have $i\in\Cval(\varkappa(\pi))$, hence $v_i=\N$.
    \item $w_i=\S$. If $i<n$, then $i\neq\pone(\pi)$, we see that $i\in\Scpk(\pi)$, meaning $i\ge\pi(i)$ and $i+1>\pi^{-1}(i+1)\neq i$. Equivalently, we have
    $$i>\pi(i)-1 \text{ and } i>\pi^{-1}(i+1),$$
    which implies that $i\in\Cpk(\varkappa(\pi))$ and so $v_i=\S$. Otherwise $i=n\neq\pone(\pi)$, we still have $v_i=\S$.
    \item $w_i=\E$. If $i<n$ and $i\neq\pone(\pi)$, then we have $i\in\Scda(\pi)$, i.e., $i<\pi(i)$ and $i+1>\pi^{-1}(i+1)$. Alternatively, we see that
    $$i\le\pi(i)-1 \text{ and } i\ge \pi^{-1}(i+1).$$
    Consequently, $i\in\Cdd(\varkappa(\pi))$ and $v_i=\E$. The other two cases $i=\pone(\pi)<n$ and $i=\pone(\pi)=n$ can be similarly checked, and in all cases we have $v_i=\E$.
    \item $w_i=\dE$. We must have $i<n$, $i\neq\pone(\pi)$, and $i\in\Scdd(\pi)$, implying that $i\ge\pi(i)$ and $i+1\le\pi^{-1}(i+1)$. Equivalently, we see that
    $$i>\pi(i)-1 \text{ and } i<\pi^{-1}(i+1),$$
    which leads to $i\in\Cda(\varkappa(\pi))$ and thus $v_i=\dE$, as desired.
\end{enumerate}
It remains to show that $b_i=c_i$ for $1\le i\le n$. Recall that $c_i=\vnest_i(\pi)+\chi(w_i=\SdE)$ and $b_i=s_{\varkappa(\pi)^{-1}(i)}+\chi(v_i=\SdE)$. So it suffices to show that $\vnest_i(\pi)=s_{\varkappa(\pi)^{-1}(i)}$.

If $i=\pone(\pi)$, we have $w_i=\NE$, $\pi(i)=1$, and $\varkappa(\pi)^{-1}(i)=n$, hence we deduce that
$$\vnest_i(\pi)=\nest_i(\pi)=0=s_n.$$
Otherwise we can assume that $i\neq\pone(\pi)$, in which case $\varkappa(\pi)^{-1}(i)=\pi(i)-1$. Viewing the relation \eqref{def:vnest} between $\vnest_i(\pi)$ and $\nest_i(\pi)$, there are four subcases to consider.
\begin{enumerate}[(i)]
    \item $\pi(i)\le i$ and $i<\pone(\pi)$. We want to show that $\nest_i(\pi)-1=s_{\pi(i)-1}$. Note that $\pi(i)-1<i$ implies $\pi(i)-1\in\Excp(\varkappa(\pi))$, so the contribution to the side number $s_{\pi(i)-1}$ comes from all those $k$, such that $k>i$, $k>\pi(k)-1$, and $\pi(k)-1<\pi(i)-1$; see the two-line notation below.
    \begin{align*}
    \varkappa(\pi)=\begin{pmatrix}
    \cdots & \pi(k)-1 & \cdots & \pi(i)-1 & \cdots & n\\
    \cdots & k & \cdots & i & \cdots & \pi^{-1}(1)
    \end{pmatrix}.
    \end{align*}
    On the other hand, the quadruple $(\pi(k),\pi(i),i,k)$ also constitutes a nesting in $\pi$ that is counted by $\nest_i(\pi)$. The term $-1$ accounts for the fact that $(1,\pi(i),i,\pone(\pi))$ contributes to $\nest_i(\pi)$ but not to $s_{\pi(i)-1}$.
    \item $\pi(i)\le i$ and $i>\pone(\pi)$. We want to show that $\nest_i(\pi)=s_{\pi(i)-1}$. The argument is similar to (i).
    \item $\pi(i)>i$ and $i>\pone(\pi)$. We want to show that $\nest_i(\pi)+1=s_{\pi(i)-1}$. Note that $\pi(i)-1\ge i$ implies $\pi(i)-1\not\in\Excp(\varkappa(\pi))$, so the contribution to the side number $s_{\pi(i)-1}$ comes from all those $k$, such that $k<i$, $k\le \pi(k)-1$, and $\pi(k)-1>\pi(i)-1$, or from $\pi^{-1}(1)=\pone(\pi)$ since $i>\pone(\pi)$; see the two-line notation below.
    \begin{align*}
    \varkappa(\pi)=\begin{pmatrix}
    \cdots & \pi(i)-1 & \cdots & \pi(k)-1 & \cdots & n\\
    \cdots & i & \cdots & k & \cdots & \pi^{-1}(1)
    \end{pmatrix}.
    \end{align*}
    On the other hand, the quadruple $(k,i,\pi(i),\pi(k))$ also constitutes a nesting in $\pi$ that is counted by $\nest_i(\pi)$. The term $+1$ accounts for the fact that $(\pone(\pi),i,\pi(i),1)$ is not a nesting while it does contribute to the side number $s_{\pi(i)-1}$.
    \item $\pi(i)>i$ and $i<\pone(\pi)$. We want to show that $\nest_i(\pi)=s_{\pi(i)-1}$. The argument is similar to (iii).
\end{enumerate}
Applying \eqref{def:vnest}, we see that in all cases we have indeed $\vnest_i(\pi)=s_{\varkappa(\pi)^{-1}(i)}$, and the proof is now completed.
\end{proof}

\section*{Acknowledgement}
The first author was supported by
the National Natural Science Foundation of China grant 11701420. The second author was supported by the National Natural Science Foundation of China grant 12171059 and the Natural Science Foundation Project of Chongqing (No.~cstc2021jcyj-msxmX0693).

\end{document}